\documentclass[10pt,letterpaper]{amsart}
\pdfoutput=1 
\usepackage{etex}
\usepackage[latin1]{inputenc}
\usepackage{multirow}
\usepackage{amsmath}

\usepackage{fullpage}
\usepackage{enumitem}
\usepackage{graphicx}
\usepackage{amssymb}
\usepackage{amsmath}
\usepackage{amsthm}
\usepackage{tikz}
\usepackage{comment}
\usepackage{amsfonts}
\usepackage{dsfont}
\usepackage{algpseudocode} 
\usepackage{algorithm}
\usepackage{hyperref}
\usepackage{mathrsfs}

\newtheorem{theorem}{Theorem}[section]

\newtheorem{problem}[theorem]{Problem}
\newtheorem{lemma}[theorem]{Lemma}

\newtheorem{corollary}[theorem]{Corollary}
\newtheorem{proposition}[theorem]{Proposition}
\theoremstyle{definition}
\newtheorem{definition}[theorem]{Definition}

\theoremstyle{remark}

\newtheorem{remark}[theorem]{Remark}

\renewcommand{\vec}[1]
{\mathbf{#1}}

\newcommand{\n}[1]
{\overline{#1}}

\newcommand{\Weights}
{\mathcal{D}_{0, n}}

\newcommand{\Weightsg}
{\mathcal{D}_{g_+, n}}

\newcommand{\Weightsgen}{\mathcal{D}_{g, n}}
\newcommand{\D}{\mathcal{D}}

\newcommand{\Famg}{\mathcal{C}_{g, \vec w}}

\newcommand{\Threshn}
{\tau_n}
\newcommand{\Thresh}{\tau}

\newcommand{\Wf}{\mathcal{W}_f}

\newcommand{\oPhi}{\overline{\Phi}}
\newcommand{\olPhi}{\overline{\Phi}_0}
\newcommand{\nrange}{{[n]}}

\newcommand{\R}{\mathbb{R}}
\newcommand{\Q}{\mathbb{Q}}

\newcommand{\N}{\mathbb{N}}



\newcommand{\GoldSD}[1]{\textproc{GoldSD}$(#1)$ }
\newcommand{\Thresf}{\textproc{Thres}$(f)$ }
\newcommand{\Goldf}{\textproc{Gold}$(f)$ }
\newcommand{\SGoldf}{\textproc{SGold}$(f)$ }

\newcommand{\IsSmalla}{\textproc{IsSmall}$(f, i, \operatorname{val})$ }
\newcommand{\SCounta}{\textproc{SCount}$(f, i, \operatorname{val})$ }
\newcommand{\GoldSDa}{\textproc{GoldSD}$(f, a)$ }

\newcommand{\val}{\operatorname{val}}

\newcommand{\real}{(\vec{w},\theta)}
\newcommand{\areal}{(\vec{w}, \theta')}
\newcommand{\vu}{\vec{u}}

\newcommand{\w}{\vec{w}}
\newcommand{\x}{\vec{x}}
\newcommand{\y}{\vec{y}}
\newcommand{\ny}{\overline{\vec{y}}}
\newcommand{\xM}{\vec{x}_M}
\newcommand{\nx}{\overline{\vec{x}}}
\newcommand{\nxM}{\overline{\vec{x}}_M}

\newcommand{\sd}{{sd}}

\newcommand{\nequiv}{\sim_N}

\newcommand{\LTF}{\operatorname{LTF}}
\newcommand{\Gold}{\operatorname{Gold_0}}

\newcommand{\SD}{\operatorname{SD}}
\newcommand{\PL}{\operatorname{PL}}
\newcommand{\Pos}{\operatorname{P}}
\newcommand{\SGold}{\operatorname{Gold_{g_+}}}

\newcommand{\calS}{\mathcal{S}}
\newcommand{\calT}{\mathcal{T}}
\newcommand{\calM}{\mathcal{M}}

\newcommand{\fp}{f^{+}}
\newcommand{\fps}{\fp_{-}}

\newcommand{\fxs}{f_{x_{i}=s}}

\title{Enumerating Hassett's wall and chamber decomposition of the moduli space of weighted stable curves}
\author{Kenneth Ascher, Connor Dub\'e, Daniel Gershenson, and Elaine Hou}
\date{}

\begin{document}

\maketitle

\begin{abstract}
Hassett constructed a class of modular compactifications of $\calM_{g,n}$ by adding weights to the marked points. This leads to a natural wall and chamber decomposition of the domain of admissible weights $\Weightsgen$, where the moduli space and universal family remain constant inside a chamber, and may change upon crossing a wall. The goal of this paper is to count the number of chambers in this decomposition. We relate these chambers to a class of boolean functions known as linear threshold functions (LTFs), and discover a subclass of LTFs which are in bijection with the chambers. Using this relation, we prove an asymptotic formula for the number of chambers, and compute the exact number of chambers for $n \leq 9$. In addition, we provide an algorithm for the enumeration of chambers of $\Weightsgen$ and prove results in computational complexity.
\end{abstract}

\section{Introduction}
\label{sec.intro}
Hassett \cite{Hassett2002} constructed a class of modular compactifications $\overline{\calM}_{g, \vec{w}}$ of the moduli space $\calM_{g, n}$ of $n$-marked curves of genus $g$ and its Deligne-Mumford compactification $\overline{\calM}_{g, n}$ by introducing a weight vector $\vec w= (w_1, \cdots, w_n) \in \mathbb{Q}^n$ with $0 < w_i \leq 1$. Hassett showed that the domain of such admissible weights admits a wall and chamber decomposition in which the moduli space and its universal family remain constant within a chamber, but may change upon crossing a wall. This raises a natural problem asked by Hassett.

\begin{problem}\cite[Problem 5.2]{Hassett2002} Find formulas for the numbers of chambers in the domain of admissible weights. \end{problem}

The goal of this paper is to address this problem, by relating chambers in Hassett's decomposition of $\Weightsgen$ (see Definitions \ref{def:dgn} and \ref{def:decomp}) to a class of boolean functions studied in computer science, known as \emph{linear threshold functions} (see Definition \ref{def:ltf}). We identify two subclasses of linear threshold functions -- which we refer to as ``Semi-Goldilocks" and ``Goldilocks" (see Definition \ref{def:sgoldilocks}) -- thus providing a combinatorial framework for the enumeration of chambers of $\Weightsgen$.

\begin{theorem}[see Corollary \ref{cor:bijection}]
Chambers in $\Weightsg$ (i.e. $\Weightsgen$ for all $g>0$) are in bijection with Semi-Goldilocks linear threshold functions. Chambers in $\Weights$ are in bijection with Goldilocks linear threshold functions. 
\end{theorem}

Ideally one would want to use this bijection to enumerate the chambers and find explicit formulas. Unfortunately, the problem of determining whether a given boolean function is a linear threshold function is co-NP-complete (see Definition \ref{conpcomplete}) \cite{hegedusconp}. In addition to developing testable criteria and algorithm for enumerating these chambers (see Appendix \ref{sec.alg}), we prove that testing whether a given boolean function is Semi-Goldilocks or Goldilocks, i.e. corresponds to a chamber in $\Weightsgen$, is equally hard (see Theorem \ref{thm.intro.hard}), and thus we do not expect an elementary formula for the number of such chambers. 

\smallskip

Instead, we determine an asymptotic formula for the number of such chambers. Here $\SGold(n)$ and $\Gold(n)$ (see Definition \ref{def:sgoldilocks}) refer to the number of chambers in the decompositions of $\Weightsg$ and $\Weights$, respectively (see Definition \ref{def:dgn}).

\begin{theorem}[see Corollaries \ref{ps.simp.asymp.cor} and \ref{!gold.asymp}]
The number of chambers for $g > 0$ and $n$ marked points satisfies the following asymptotic formula:
$$ \SGold(n) \sim 2^{n^2 - n \log_2 n + O(n)}.$$ 

\noindent The number of chambers for $g = 0$ and $n$ marked points satisfies the following asymptotic formula:
$$ \Gold(n) \sim \frac{\SGold(n)}{2}.$$
\end{theorem}

We also obtain numerical results for $n \leq 9$:

\begin{center}
\begin{figure}[!htbp]
\caption{The number of chambers of $\Weightsg$ for small $n$}
\label{table.dgn}
\begin{tabular}{|c||c|c|}
\hline
$n$ & $\SGold(n)$ & $\SGold(n)/S_n$\\
\hline
1& 1& 1\\
2& 2& 2\\
3& 9& 5\\
4& 96& 17\\
5& 2690& 92\\
6& 226360& 994\\
7& 64646855& 28262\\
8& 68339572672& 2700791\\
9& 281196831947304& 990331318 \\
\hline
\end{tabular}
\end{figure}
\end{center}

\begin{center}
\begin{figure}[!htbp]
\caption{The number of chambers of $\Weights$ for small $n$}
\label{table.d0n}
\begin{tabular}{|c||c|c|}
\hline
$n$ & $\Gold(n)$ & $\Gold(n)/S_n$\\
\hline
3& 1& 1\\
4& 27& 5\\
5& 1087& 36\\
6& 105123& 448\\
7& 31562520& 13642\\
8& 33924554539& 1336943\\
9& 140306938682875& 493888290 \\
\hline
\end{tabular}
\end{figure}
\end{center}
We present both the number of chambers as well as the number of chambers up to the natural action of $S_n$.\\
 
Furthermore, we make precise the difficulty of enumerating chambers of $\Weightsgen$ as follows.

\begin{theorem}[see Theorems \ref{coNPcomp.sgold.thm} and \ref{coNPcomp.thm}]\label{thm.intro.hard}
Given an arbitrary boolean function $f$, the problems of determining whether $f$ is a Semi-Goldilocks or Goldilocks function are co-NP-complete.
\end{theorem}

Finally, we note that Song \cite{song} began studying this problem using the theory of hyperplane arrangements. In the end Song modified their problem to be better suited for a graph theoretic approach and did not obtain results for Hassett's problem.

\subsection{Outline}
In Section \ref{sec:ag} we motivate the problem and recall the definition of the domain of admissible weights from \cite{Hassett2002}. In Section \ref{sec:ltfs} we state definitions and main properties of linear threshold functions, introduce our subclasses of linear threshold functions known as ``Semi-Goldilocks" and ``Goldilocks" linear threshold functions, and prove that these are in bijection to chambers of $\Weightsgen$. 

Section \ref{sec:properties} serves to demonstrate further properties of Semi-Goldilocks and Goldilocks LTFs that are necessary for both the proofs of our asymptotics as well as our algorithm for counting chambers. The purpose of Section \ref{sec:nondeg} is to motivate and prove Lemma \ref{ltf.nondegen.formula.lem}, which is instrumental in Section \ref{sec.asymptotics} where we compute asymptotics. In Section \ref{sec.asymptotics} we introduce known asymptotics for linear threshold functions and apply properties of linear threshold functions to obtain asymptotic formulas for the number of chambers. In Section \ref{sec.hardness} we show that deciding whether a boolean function is Semi-Goldilocks or Goldilocks is co-NP-complete.

Finally, in Appendix  \ref{sec.alg} we present our algorithm for enumerating the chambers of $\Weightsgen$ and prove its validity.

\subsection*{Acknowledgements}
This research was completed as part of the SUMRY (Summer Undergraduate Mathematics Research at Yale) program during the summer of 2017, where the first author was a mentor and the other authors were participants. SUMRY is supported in part by NSF Grant CAREER DMS-1149054. K.A. was supported in part by an NSF Postdoctoral Fellowship.

We thank Dori Bejleri, Patricio Gallardo, Brendan Hassett, Dave Jensen, Steffen Marcus, Sam Payne, and Dhruv Ranganathan for helpful discussions and suggestions. We thank Nicolle Gruzling for providing us with code used to enumerate linear threshold functions for $n \leq 9$. Finally, we thank the referee for many usual suggestions that helped improve this paper.

\section{Hassett's chamber decomposition $\Weightsgen$} \label{sec:ag}
The moduli space $\calM_{g,n}$ parametrizes smooth curves of genus $g$ with $n$ marked points. Its Deligne-Mumford compactification $\overline{\calM}_{g,n}$ parametrizes \emph{stable} curves $(C, p_1 + \cdots + p_n)$ of genus $g$ with $n$ marked points. Here stable means: \begin{enumerate} \item  $C$ has at worst nodal singularities, \item the points $p_1, \cdots, p_n$ are distinct, and \item the log canonical bundle $\omega_C(p_1 + \cdots + p_n)$ is ample. \end{enumerate}

\medskip

As mentioned in the introduction, Hassett \cite{Hassett2002} constructed a large class of modular compactifications of $\calM_{g,n}$ by adding weights to the marked points. In particular, consider $\overline{\calM}_{g, \vec w}$, where $\vec w = (w_1, \cdots, w_n) \in \mathbb{Q}^n$ with $0 < w_i \leq 1$. These moduli spaces parametrize \emph{weighted} stable pointed curves. Stability now means: \begin{enumerate} \item $C$ has at worst nodal singularities, \item points $p_1, \cdots, p_n$ can collide as long as their total weight is $\leq 1$, and \item the log canonical bundle $\omega_C(w_1p_1 + \cdots + w_n p_n)$ is ample. \end{enumerate}

\medskip

Note that if $\vec w = (1, \cdots, 1)$ then $\overline{\calM}_{g, \vec w} \cong \overline{\calM}_{g,n}$.

\begin{definition}\label{def:dgn}
The space of admissible weights $\Weightsgen \subseteq \R^n$  is defined as \begin{equation*}
\Weightsgen := \{(w_1,w_2, \ldots, w_n) \in \mathbb{R}^n: 0 < w_j \leq 1 \;\,\forall j\text{ and } w_1 + w_2 + \ldots + w_n > 2 - 2g\}. 
\end{equation*} 

Note that as long as $n > 0$, the space $\Weightsgen$ is identical for all $g > 0$. We denote the space of admissible weights where $g = 0$ by $\Weights$, and when $g > 0$ by $\Weightsg$. 
\end{definition}

Note that both of these domains are convex polyhedra in $\R^n$.

\begin{definition}\label{def:decomp}
Given a convex polyhedron $\D$, a \textit{chamber decomposition} of $\D$ consists of a finite set $\mathcal{W}$ of hyperplanes, which we call the \textit{walls} of the chamber decomposition. The \textit{chambers} of the decomposition are the connected components of the complement of the union of the walls with respect to $\D$. \\

Consider the convex polyhedron $\Weightsgen$. The \textit{fine chamber decomposition} is described by the walls
\[
	\Wf = \{\sum_{j\in S} w_j = 1:S\subseteq 	\nrange, \textit{ } S \neq \emptyset\}, 
\]
where $\nrange := \{1, \ldots, n\}$.
\end{definition}

Hassett proved the following theorem regarding the fine chamber decomposition. 

\begin{theorem}\label{CoarsestDec}\cite[Proposition 5.1]{Hassett2002}
The fine chamber decomposition is the coarsest chamber decomposition of $\Weightsgen$ such that $\Famg$ is constant on each chamber, where $\Famg \to \overline{\calM}_{g, \vec w}$ is the universal family. 
\end{theorem}

\begin{remark}
$ $
\begin{enumerate}
\item Since $\Weightsgen$ is identical for all $g >0$, there are two combinatorial objects of study: the fine chamber decomposition of $\Weights$ and the fine chamber decomposition of $\Weightsg$. Furthermore, these can be viewed as the restriction of the same chamber decomposition of $\R^n$ to two distinct polyhedra.
\item Our definition of $\Wf$ differs slightly from the one presented in \cite{Hassett2002}. We note that in the $g = 0$ case they coincide, and in the $g > 0$ case we add walls that were missing from \cite{Hassett2002}. In particular, when $g > 0$ one must consider walls of size $|S| \le n$. It is only in the $g = 0$ case where fine walls must satisfy $|S| \le n-2$. 
\end{enumerate}
 \end{remark}

The purpose of this paper is to enumerate the fine chambers (or ``chambers") of $\Weights$ and $\Weightsg$. 

\begin{figure}[!htbp]
\begin{center}
\begin{tikzpicture}[scale = 0.5]
\draw [dashed, fill=gray] (3,8) -- (6,6) -- (9,2) -- (3,8);
\draw [fill=gray, opacity=.3] (3,8) -- (9,8) -- (9,2) -- (4.5,4) -- (3,8);
\draw (3,2) -- (0,0) -- (6,0) -- (9,2) -- (9,8) -- (6,6) -- (0,6) -- (3,8) -- (3,2) -- (9,2) -- (6,0) -- (6,6) -- (9,8) -- (3,8) -- (0,6) -- (0,0);
\draw [black] (6,6) -- (9,2) -- (3,2) -- (6,6);
\draw [black] (3,2) -- (6,0) -- (6,6) -- (3,8) -- (3,2);
\draw [black] (0,6) -- (6,6) -- (3,2) -- (0,6);
\draw [black] (6,0) -- (0,6) -- (3,2) -- (6,0);
\draw (6,0) -- (3,8);
\draw (0,6) -- (9,2);

\end{tikzpicture}

\caption{Chamber decomposition of $\mathcal D_{g,3}$. The dark gray hyperplane represents $x + y + z = 2$, and the shaded chamber represents the only fine chamber on $3$ variables.}
\label{LTF.chamber.diag}
\end{center}
\end{figure}
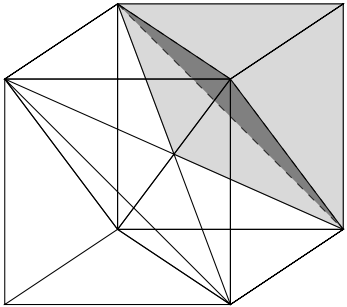

\section{A bijection from chambers of $\Weightsgen$ to linear threshold functions}\label{sec:ltfs}
The goal of this section is to construct a bijection from chambers of $\Weightsgen$ to linear threshold functions, which we use to provide numerical results for $n \leq 9$. We begin by providing properties of linear threshold functions and then determine subclasses of linear threshold functions which have a natural bijection to chambers of $\Weightsgen$.
\subsection{Linear threshold functions}
We first introduce linear threshold functions and many of their basic properties. For proofs of many of the theorems stated here, we direct those interested to the original papers as well as to \cite{gruzling08}. 

\begin{definition}\label{def:ltf}
A \emph{linear threshold function} (or LTF or threshold function) in $n$ variables is a function $f:\{0, 1\}^n \mapsto \{0, 1\}$ that can be expressed as 
\[
	f(\vec{x}) = \operatorname{sgn}(\vec{w} \cdot \vec{x} - \theta), 
\]
for some $\vec{w} \in \Q^n$ and $\theta \in \Q$.  A particular $(\vec{w}, \theta)$ pair is called a \emph{realization} of $f$, where $\vec{w}$ is called the \emph{weight vector} and $\theta$ is called the \emph{bias}. \end{definition}

We use the convention that $\operatorname{sgn}$ takes a value of $1$ on positive arguments and $0$ on nonpositive arguments, although we note that it is always possible to choose a realization which avoids the ambiguity of $\operatorname{sgn}(0)$ \cite[Section 3.2.1]{anthonyThres}. Linear threshold functions are also called \emph{separable} boolean functions. Finally, we call the space of $n$-variate linear threshold functions $\Threshn$.

\begin{remark}
Note that a single LTF will have infinitely many realizations. 
\end{remark}

\begin{definition}
For any two LTFs in $n$ variables $f$ and $h$, we say that $f \geq h$ if for all $\x \in \{0, 1\}^n$, we have $f(\x) \geq h(\x)$.
\end{definition}

\begin{theorem}\cite[Theorem 1, Theorem 5]{Winder61}
All linear threshold functions are \emph{unate}, that is, for each variable $x_i$, they are increasing in either $x_i$ or its negation.
\end{theorem}

\begin{definition}
We denote by $T_f$ the \emph{truth set} of $f$ and by $F_f$ the \emph{false set} of $f$. That is, 
\[
	T_f := f^{-1}(1) \quad \text{and}\quad F_f := f^{-1}(0).
\]
If there is no ambiguity, we use $T$ and $F$ to denote these sets.
\end{definition}

\begin{definition}
For any linear threshold function $f$, we define the \emph{dual linear threshold function} $f^d$ by:
\[
	f^d(\vec{x}) := \overline{f(\nx)}, 
\]
where $\nx = (\n{x_1}, \ldots, \n{x_n})$, that is, the negation of $\vec{x}$. If $f = f^d$ we call $f$ \emph{self-dual}.
\end{definition}

\begin{remark}\label{sd.01pairs.rem}
An LTF $f$ is self-dual if and only if every negation pair $\x, \nx \in \{0, 1\}^n$ satisfies $f(\x) \neq f(\nx)$.
\end{remark}

The following characterization of linear threshold functions is frequently useful.

\begin{theorem} \cite[Theorem 2]{gruzling08}\label{asum.thm}
Let $f$ be a boolean function with true set $T = \{\x_1, \ldots, \x_m\}$ and false set $F = \{\x_{m+1}, \ldots, \x_{2^n} \}$. Then $f$ is a linear threshold function if and only if for any set of natural numbers $c_i \ge 0$ (with $1 \le i \le 2^n$), the equalities
\[
	\sum_{i =1}^{m} c_i = \sum_{i = m+1}^{2^n} c_i \quad \text{and}\quad \sum_{i = 1}^m c_i \x_i = \sum_{i = m+1}^{2^n} c_i \x_i
\]
imply that $c_i = 0$ for all $i$.
\end{theorem}
\begin{definition} We say that such a function satisfying Theorem \ref{asum.thm} satisfies the \emph{asummability criterion}. \end{definition} This is most frequently used through the following corollary. 
\begin{corollary}\label{simpbad.cor}
If there exist two negation pairs $\x, \nx$ and $\y, \n{\y}$ in $\{0, 1\}^n$ such that $f(\x) = 0 = f(\nx)$ and $f(\y) = 1 = f(\n{\y})$, then $f$ is not a linear threshold function.
\end{corollary}

\subsection{The correspondence}\label{correspondence}
We associate linear threshold functions with weights $\vec{w} \in \Weightsgen$ by virtue of the following construction.

\begin{definition} For any $\vec{w} \in \Weightsgen$, define the corresponding threshold function $f_\vec{w}$ to be the threshold function given by representation $(\vec{w}, 1)$. That is, 
\[
	f_\vec{w}(\vec{x}) := \operatorname{sgn}(\vec{w} \cdot \vec{x} - 1).
\]
\end{definition}
Let $\Phi : \Weightsg \to \Threshn$ be the map $\Phi(\vec{w}) = f_\vec{w}$, where $\Weightsg$ is the domain of admissible weights for $g > 0$. Similarly, let $\Phi_0 : \Weights \to \Threshn$ be the map $\Phi_0(\w) = f_\w$. 

\begin{remark}
Notice that when defining chambers of $\Weightsgen$, the set of chambers for genus $0$ is a subset of chambers for genus $g>0$, because the stability condition is trivial when $g > 0$ and $n > 0$. Therefore, $\Phi_0$ is the restriction of $\Phi$ to $\Weights \subseteq \Weightsg$. 
\end{remark}

We will prove that $\Phi$ and $\Phi_0$ induce bijections between chambers of $\Weightsgen$ and certain threshold functions. 

\begin{proposition}
Given any two weights $\vec{w}, \vec{w}' \in \Weightsgen$, then the threshold functions $f_\vec{w}, f_{\vec{w}'}$ are equal if and only if they are contained in the same chamber of $\Weightsgen$.
\begin{proof}
Recall that the bounding hyperplanes of the chamber decomposition of $\Weightsgen$ are given by $\sum_{i \in S} x_i = 1$ for all $S \subseteq \nrange$.

For one direction, assume that $\vec{w}, \vec{w}'$ are in different chambers. Then there is at least one boundary hyperplane between them. This hyperplane is defined by $S$ for some $S \subseteq \nrange$. Without loss of generality, we can assume $\vec{w}$ is in the (closed) half-space which contains the origin, and $\vec{w}'$ in the (open) other half. Then by definition,  
\[
	\sum_{i \in S} w_i \le 1 \quad \text{and}\quad \sum_{i \in S} w_i' > 1.
\]
Let $\vec{x} \in \{0, 1\}^n$ be the indicator function for $S$, that is, the vector which is $1$ at index $i \in \nrange$ if $i \in S$ and $0$ otherwise. Consider the threshold functions $f_\w, f_{\w'}$ evaluated at $\vec{x}$. Then $f_\vec{w}(\vec{x}) = 0$ by the inequality above, while $f_{\vec{w}'}(\vec{x}) = 1$.

Conversely, assume that on some vector $\vec{x}$, the functions evaluate as $f_\w(\vec{x}) = 0$ and $f_{\w'}(\vec{x}) = 1$. Let $S = \{i \;| \; x_i = 1 \}$. Then the inequalities $\sum_{i \in S} w_i \le 1$ and $\sum_{i \in S} w_i' > 1$ follow from the definitions. Thus the two points are separated by the boundary hyperplane $\sum_{i \in S} x_i = 1$, and therefore they are in distinct chambers.
\end{proof}
\end{proposition}

This implies immediate corollaries.
\begin{corollary}
Let $\sim$ be the equivalence relation on $\Weightsg$ and $\Weights$ defined as follows: we say that $\vec{w} \sim \vec{w}'$ if and only if $\vec{w}$ and $\vec{w}'$ are contained in the same chamber. Then the quotient maps $\oPhi : \Weightsg/\sim \; \to \Threshn$ induced by $\Phi$ and $\olPhi : \Weights / \sim \; \to \Threshn$ induced by $\Phi_0$ are well-defined and injective.
\end{corollary}

\begin{corollary}
The number of threshold functions is an upper bound on the number of chambers of $\Weightsgen$.
\end{corollary}

Note that each equivalence class of $\sim$ is a chamber of $\Weightsgen$. 

\subsection{Semi-Goldilocks and Goldilocks LTFs} Our goal is to characterize the images of $\oPhi$ and $\olPhi$ and thereby induce a bijection from chambers onto subclasses of threshold functions. The number of LTFs is greater than the number of chambers of $\Weightsgen$, as chambers are subject to additional constraints. In particular, the weight of each marked point is strictly positive but $\leq 1$. Moreover, in the $g = 0$ case, we have an ampleness requirement that $\sum w_i > 2$. The goal of this section is to determine a subclass of LTFs that take these requirements into account, so that we can pursue Hassett's question.  The relevant definitions are given below.

\begin{definition}\label{def:sgoldilocks}
We call a threshold function $f$ \emph{Semi-Goldilocks} if it satisfies the following criteria: 
\begin{enumerate}[label=(\roman*)]
\item (\emph{Positivity}) There exists a realization $(\vec{w}, \theta)$ such that $w_i > 0$ for all $i$.
\item (\emph{Smallness}) There exists a realization $(\vec{w}', \theta)$ such that $w'_i \le \theta$ for all $i$.
\end{enumerate}
We refer to a realization with any of these properties by the corresponding name and call a single realization with both properties a \emph{Semi-Goldilocks realization}. We define $\SGold(n)$ to be the number of $n$-variate Semi-Goldilocks linear threshold functions.

We call a threshold function $f$ \emph{Goldilocks} if it is a Semi-Goldilocks function that satisfies the following criterion:
\begin{enumerate} \item[(iii)] (\emph{Ampleness}) There exists a realization $(\vec{w}'', \theta)$ such that $\sum_{i = 1}^n w''_i > 2\theta$. \end{enumerate}

We call a single realization with all three properties (positivity, ampleness, smallness) a \emph{Goldilocks realization}. We define $\Gold(n)$ to be the number of $n$-variate Goldilocks linear threshold functions. 

\begin{remark} The term ``Goldilocks'' was chosen to emphasize the fact that such functions are neither too big (that is, they are \emph{small}), nor too little (because they are \emph{ample}), but sit in a limited intermediate region. \end{remark}

\end{definition}

\begin{remark}
The notation $\SGold(n)$ refers to the number of Semi-Goldilocks functions, which we relate to the number of chambers $\Weightsg$, and the notation $\Gold(n)$ refers to the number of Goldilocks functions, which we relate to the number of chambers in $\Weights$. The subscripts refer to whether the genus of the corresponding domain is positive. 
\end{remark}

We prove that the image of $\olPhi$ is exactly the set of Semi-Goldilocks linear threshold functions and that the image of $\oPhi$ is exactly the set of Goldilocks linear threshold functions. Before doing so, we state some necessary propositions.

\begin{proposition}\label{smallreal}
Every realization of a small threshold function is small.
\begin{proof}
If an LTF $f$ is small, then it satisfies $f(\hat{e}_i) = 0$ for all $i$. Any pair $(\vec{w}, \theta)$ which is not small has $w_i > \theta$ for some $i$, so $f(\hat{e}_i) = 1$ for that $i$, and thus cannot be a realization for $f$. 
\end{proof}
\end{proposition}

\begin{proposition}\label{ampletocomb}
If $f$ is ample, then for any $\vec{x} \in \{0, 1\}^n$ with $f(\x) = 0$, we have $f(\nx) = 1$.
\begin{proof}
The proof proceeds by contrapositive. Let $\real$ be any realization for $f$. Assume that there exists a negation pair $\x, \nx$ such that $f(\x) = 0 = f(\nx)$. Then we have both $\w \cdot \x \le \theta$ and $\w \cdot \nx \le \theta$, so 
\[
	\sum_{i = 1}^n w_i = \w \cdot \x + \w \cdot \nx \le 2\theta.
\]
Thus an arbitrary realization $\real$ for $f$ is not ample, so $f$ is not ample.
\end{proof}
\end{proposition}

\begin{proposition}[Amplification]\label{amplification}
Given any realization $(\vec{w}, \theta)$ of an ample linear threshold function $f$, there exists an ample realization $(\vec{w}, \theta')$ for $f$, called the \emph{amplification} of $(\vec{w}, \theta)$, with the same weight vector and satisfying $\theta' \le \theta$.
\begin{proof}
Let $(\vec{w}, \theta)$ be any realization of an ample $f$. If $\real$ is already ample, let $\theta' = \theta$ and the result is trivial. Thus we assume that $\real$ is not ample. 

Define the ``false set'' $F$ on $f$:
\[
	F := \{ \vec{x} \in \{0, 1\}^n \; | \; f(\vec{x}) = 0 \}.
\]
Let $\vec{x}_M$ be a vector in $F$ that maximizes $\vec{w} \cdot \vec{x}$ over $F$. Thus, any other $\x \in F$ satisfies $\w \cdot \x \le \w \cdot \xM \le \theta$. Since $\real$ is assumed not to be ample, we have the following inequality:
\[
	\w \cdot \xM + \w \cdot \nxM = \sum_{i = 1}^n w_i \le 2\theta.
\]
Define $\delta \ge 0$ to be the slack variable on this inequality: it is the nonnegative rational number such that 
\[
	\w \cdot \xM + \w \cdot \nxM = 2\theta - \delta.
\]
Since $\nxM$ is the negation of a vector in $F$, it itself must satisfy $f(\nxM) = 1$ by Proposition \ref{ampletocomb} and thus must satisfy $\w \cdot \nxM > \theta$. Thus we have from the above equality $\w \cdot \xM < \theta - \delta$. Let $\epsilon > 0$ be the slack in this equality, so that $\w \cdot \xM = \theta - \delta - \epsilon$. 

Let $\theta' = \theta - \delta - \epsilon$. We claim that $\areal$ is a realization for $f$. If $f(\x) = 0$, then $\x \in F$, so we have
\[
	\w \cdot \x \le \w \cdot \xM = \theta'
\]
and thus the new realization preserves the value of $\x$. Alternatively, if $f(\x) = 1$, then $\w \cdot \x > \theta > \theta'$, so $\operatorname{sgn}(\w \cdot \x - \theta')$ agrees with $f$ on all inputs and $\areal$ is a realization for $f$. Furthermore, we have
\[
	\sum_{i = 1}^n w_i = \w \cdot \xM + \w \cdot \nxM > \theta + \theta' > 2\theta'
\]
and thus $\areal$ is ample.
\end{proof}
\end{proposition}


The critical lemma is the following.

\begin{lemma}\label{Goldrep}
A threshold function $f$ is Semi-Goldilocks if and only if it has a Semi-Goldilocks realization. Similarly, a threshold function $f$ is Goldilocks if and only if it has a Goldilocks realization. 

\begin{proof}
Clearly, a function with a Semi-Goldilocks realization is Semi-Goldilocks, but the converse is nontrivial. Assume that $f$ is Semi-Goldilocks, and let $(\vec{w}, \theta)$ be a positive realization of $f$. By Proposition \ref{smallreal}, $(\vec{w}, \theta)$ is a realization of a small threshold function, and is thus itself small. Thus $(\vec{w}, \theta)$ is Semi-Goldilocks.

Now, a function with a Golilocks realization is clearly Goldilocks, so we proceed to prove the converse. By the above, a Goldilocks function, which is Semi-Goldilocks, has a Semi-Goldilocks realization $(\w, \theta)$. If $(\w, \theta)$ is ample, then it is a Goldilocks realization of $f$. If $(\w, \theta)$ is not ample, let $(\w, \theta')$ be the corresponding amplification, which exists by Proposition \ref{amplification}. Note that $(\w, \theta')$ is a positive and ample realization of a small linear threshold function, so it is small as well. Thus $(\w, \theta')$ is Goldilocks. 
\end{proof}
\end{lemma}

\begin{theorem}\label{bijection}
A threshold function $f$ is in the image of $\oPhi$ if and only if it is Semi-Goldilocks. Similarly, a threshold function $f$ is in the image of $\olPhi$ if and only if it is Goldilocks. 
\begin{proof}
We remark that a threshold function is in the image of $\oPhi$ or $\olPhi$ if and only if it has a \emph{Semi-Goldilocks realization} or \emph{Goldilocks realization}, respectively. If a threshold function $f$ is in $\operatorname{Im} \oPhi$ (resp. $\operatorname{Im} \olPhi$), then there is some $\vec{w}$ such that $f = f_\vec{w}$, and $(\vec{w}, 1)$ is a Semi-Goldilocks (resp. Goldilocks) realization for $f$. Conversely, if $f$ has a Semi-Goldilocks (resp. Goldilocks) realization $(\vec{w}, \theta)$, then scaling by dividing $\theta$ through all the inequalities gives a normalized $\vec{w}$ which satisfies the properties to be in $\Weightsg$ (resp. $\Weights$). Indeed, if $\theta = 0$, then any small realization with $\theta = 0$ has $\vec{w} = 0$ and thus cannot be positive. Thus a threshold function is in $\operatorname{Im}\oPhi$ (resp. $\operatorname{Im} \olPhi$) if and only if it has a Semi-Goldilocks (resp. Goldilocks) realization, and the theorem follows from Lemma \ref{Goldrep}.
\end{proof}
\end{theorem}

Thus $\oPhi$ and $\olPhi$ biject chambers of $\Weightsgen$ with subclasses of threshold functions. 

\begin{remark}
Note that $\oPhi$ (resp. $\olPhi$) biject chambers in the fine chamber decomposition with Semi-Goldilocks (resp. Goldilocks) functions. It is clear that an LTF with an associated weight is a necessary and sufficient invariant for identifying a chamber in $\Weightsg$ (resp. $\Weights$) given by $\sum_{i \in S} w_i = 1$, where $|S|$ is allowed to range from $1$ to $n$. In the definition of the fine chamber decomposition (Definition \ref{def:decomp}), the decomposition is the restriction of a subset of $\R^n$ to the convex polytope defined by
\[
	\{\w \in \Q^n:0 < w_i < 1, \sum w_i > 2 - 2g\}, 
\]
where we can assume inequalities are strict since we are considering chambers. A chamber in $\R^n$ is in the fine decomposition if and only if it has a nontrivial intersection with the polytope. 

If an LTF is Semi-Goldilocks (resp. Goldilocks), its corresponding chamber intersects the interior of the fine decomposition (a Semi-Goldilocks or Goldilocks realization lies within the polytope). Conversely, a chamber intersecting the fine decomposition has a Semi-Goldilocks (resp. Goldilocks) realization, so it must correspond to a Semi-Goldilocks (resp. Goldilocks) LTF. 
\end{remark}

We have the following result on the number of chambers in the fine decomposition of $\Weightsgen$. 

\begin{corollary}\label{cor:bijection}
The number of Semi-Goldilocks threshold functions is the number of chambers of $\Weightsg$ (that is, in $\Weightsgen$ for all $g > 0$). Similarly, the number of Goldilocks threshold functions is the number of chambers of $\Weights$. \end{corollary}

In Sections \ref{comb criterion} and \ref{sec:chowgold}, we prove results on these criteria which are used throughout the rest of the paper. Using these results, we create an algorithm that enumerates chambers of $\Weightsgen$ for $n \leq 9$. Algorithm \ref{algorithm1} (see Appendix \ref{sec.alg}) was implemented in C++ and run for small $n$. Figures \ref{table.dgn} and \ref{table.d0n} (see Section \ref{sec.intro}) enumerate the number of chambers of $\Weightsgen$ for $n \leq 9$ as well as the quotient of the number of chambers of $\Weightsgen$ under the action of $S_n$.\\

Since the problem of deciding whether an arbitrary boolean function is Semi-Goldilocks or Goldilocks is co-NP-complete (see Section \ref{sec.hardness}), there likely do not exist elementary formulas for the number of chambers of $\Weightsgen$. We instead focus on finding asymptotic formulas for the growth of the number of such chambers.

\section{Properties of Semi-Goldilocks, Goldilocks, and linear threshold functions}
\label{sec:properties}

Here we demonstrate those properties of Semi-Goldilocks and Goldilocks functions necessary to enumerate chambers of $\Weightsgen$ by virtue of the correspondence to LTFs. In Sections \ref{chow parameters}-\ref{selfdual}, we introduce the remaining linear threshold function theory from the literature necessary for our purposes. In Sections \ref{comb criterion}-\ref{sec:chowgold}, we develop a theory of Semi-Goldilocks and Goldilocks functions in terms of linear threshold function theory. 

\subsection{Chow parameters}\label{chow parameters}


The study of linear threshold functions owes a great deal to a set of associated invariants introduced in an original form by Golomb \cite{golomb59}, but first refined and systematically studied by Chow \cite{chow61}. There is still some variance in how these parameters are defined; we adopt the following formulation.
\begin{definition}
Let $f$ be a threshold function and $T \subseteq \{0, 1\}^n$ the true set of $f$. The \emph{Chow parameters} are a pair $(m_f, \vec{a}_f)$ with $m_f \in \N, \vec{a}_f \in \N^n$, defined as:
\[
	m_f := |T| \quad \text{and} \quad \vec{a}_f := \sum_{\vec{x} \in T} \vec{x} + \sum_{\vec{x} \in F} \nx.
\]
Thus, $m_f$ is the number of ``true'' vectors in the domain, and the $i$-th component of $\vec{a}_f$ is the number of true vectors which are 1 in the $i$-th entry, plus the number of false vectors which are 0 in the $i$-th entry. If there is no confusion, we denote the Chow parameters as $(m, \vec{a})$ or $(a_0, \ldots, a_n)$.
\end{definition}

Chow proves several theorems on these parameters.

\begin{theorem}\cite[Theorem 1]{chow61}
Let $f, h$ be two $n$-variable boolean functions with the same Chow parameters $(m, \vec{a})$. Either both $f$ and $h$ are linear threshold functions and $f = h$, or neither are linear threshold functions.
\end{theorem}

\begin{corollary}\cite[Theorem 2]{chow61}
\label{chowunique}
The Chow parameters uniquely characterize a linear threshold function.
\end{corollary}

Furthermore, the Chow parameters have clear ramifications on realizations of the corresponding linear threshold function.
\begin{theorem}\cite[Lemma 1]{chow61}\label{chowclassify}
Let $f$ be a linear threshold function with Chow parameters $(m, \vec{a})$. The following properties hold:
\begin{itemize}
\item $a_i > 2^{n-1}$ if and only if all realizations of $f$ have $w_i > 0$,
\item $a_i < 2^{n-1}$ if and only if all realizations of $f$ have $w_i < 0$,
\item $a_i = 2^{n-1}$ if and only if there exists a realization of $f$ with $w_i = 0$,
\item $a_i > a_j$ if and only if $w_i > w_j$ for all realizations of $f$,
\item $a_i = a_j$ if and only if there exists a realization of $f$ with $w_i = w_j$.
\end{itemize}
We call those variables $x_i$ with $a_i = 2^{n-1}$ \emph{weak variables} (or $\epsilon$ variables) for $f$.
\end{theorem}

The dual behaves well with this set of invariants.

\begin{proposition}\label{chow.dual.prop}
If an LTF $f$ on $n$ variables has Chow parameters $(a_0, \ldots, a_n)$, then $f^d$ has Chow parameters $(2^n - a_0, a_1, \ldots, a_n)$.
\end{proposition}

\subsection{Self-dualization and equivalence classes of threshold functions}\label{selfdual}
In order to make the space of linear threshold functions more tractable, equivalence classes of linear threshold functions are often studied. A more thorough exposition can be found in \cite{gruzling08}.

\begin{definition}
For any $\vec{u} \in \{0, 1\}^n$, we define the \emph{$\vec{u}$-complementation}, or \emph{complementation by $\vec{u}$}, as the operator $\gamma_\vec{u}$ on $\{0, 1\}^n$ which negates the variables with $u_i = 1$ and preserves the other variables.
\end{definition}

\begin{definition}
Consider the equivalence relation $\nequiv$ on linear threshold functions in which $f \nequiv h$ if there exists some $\vec{u} \in \{0, 1\}^n$ for which $f \circ \gamma_\vec{u} = h$. We define the \emph{N-type classes} of linear threshold functions as the equivalence classes of $\nequiv$. 
\end{definition}

The Chow parameters behave well under $\vu$-complementation.

\begin{proposition}\label{chow.ucomp.prop}
Given an LTF $f$ with Chow parameters $(m, \vec{a})$, the $\vu$-complementation of $f$ has Chow parameters $(m, \vec{b})$, where
\[
	b_i = \begin{cases}
    	a_i \quad \text{if $u_i = 0$}\\
        2^n - a_i \text{if $u_i = 1$}. 
    \end{cases}
\]
\end{proposition}

The equivalence relation can be extended to include permutations of the arguments as well.

\begin{definition}Consider the equivalence relation $\sim_{NP}$ such that $f \sim_{NP} h$ if there exists some permutation $\sigma$ of the arguments of $f$ and some $\vec{u}$ such that 
\[
	f \circ \gamma_{\vec{u}} \circ \sigma = h.
\]
The \emph{NP-type classes} of linear threshold functions are defined as the $\sim_{NP}$ equivalence classes of $f$.
\end{definition}

\begin{proposition}\label{chow.perm.prop}
Given an LTF $f$ with Chow parameters $(m, \vec{a})$ and a permutation $\sigma \in S_n$ the Chow parameters for $f\circ \sigma$ are given by $(m, \vec{b})$, where $b_i = a_{\sigma(i)}$. 
\end{proposition}

\begin{corollary}\label{chow.perm.fixed.prop}
A permutation of arguments fixes an LTF $f$ if and only if it fixes the Chow parameters for $f$.
\begin{proof}
Let $f$ have Chow parameters $(m, \vec{a})$ and fix a permutation $\sigma$. Then by Proposition \ref{chow.perm.prop}, the Chow parameters of $f \circ \sigma$ are given by $(m, a_{\sigma(1)}, \ldots, a_{\sigma(n)})$. Since the Chow parameters identify the LTF, $f \circ \sigma = f$ if and only if $a_i = a_{\sigma(i)}$ for all $i$.
\end{proof}
\end{corollary}

There is a much stronger notion of equivalence on LTFs, however, motivated by the self-dualization construction.

\begin{definition}\label{sd.class.def}
Given a boolean function $f$ on $n$ variables, define the \emph{self-dualization} of $f$ as the unique boolean function on $n+1$ variables such that
\[
	f^\sd(x_0, \ldots, x_n) := \nx_0 f(x_1, \ldots, x_n) + \x_0 f^d(x_1,\ldots, x_n).
\]
This process can be reversed: given a self-dual boolean $\mathscr{F}$ on $n+1$ variables, the \emph{anti-self-dualization} $f$ of $\mathscr{F}$ is given by
\[
	f(x_1, \ldots, x_n) := \mathscr{F}(0, x_1, \ldots, x_n).
\]
\end{definition}

The self-dual is well behaved with Chow parameters.
\begin{proposition}\label{chow.sdual.prop}
If an LTF $f$ on $n$ variables has Chow parameters $(m, a_1, \ldots, a_n)$, then $f^\sd$ has Chow parameters $(2^{n}, 2^{n+1} - 2m, 2a_1, \ldots, 2a_n)$.
\end{proposition}

\begin{definition}
Two boolean functions $f, h$ are said to be in the same \emph{SD-type class}, or \emph{SD class}, if there exists any sequence of permutations, negations, self-dualizations, and anti-self-dualizations by which $f$ can be transformed into $h$.
\end{definition}

The merit of this classification is justified by the following theorems.

\begin{theorem}\cite[Theorem 2]{Goto62}
If one element of an SD class is a linear threshold function, every element of the class is a linear threshold function.
\end{theorem}

\begin{theorem}\cite[Theorem 1]{Goto62}
\label{sd.bijection}
The self-dualization map is bijective: the number of linear threshold functions on $n$ variables is exactly the number of self-dual linear threshold functions on $n+1$ variables.
\end{theorem}

\subsection{A combinatorial criterion for Goldilocks}\label{comb criterion}

We now introduce a set of equivalent definitions for the Goldilocks criteria with a more combinatorial flavor. We do not separately discuss Semi-Goldilocks functions here since they satisfy a subset of the same criteria.

\begin{theorem}\label{equivGold}
The criteria of Goldilocks are equivalent to the following:
\emph{
\begin{enumerate}[label=(\roman*')]
\item (Combinatorial Positivity) For all $i \in \nrange$ and all $x_j \in \{0, 1\}$,
\[
	f(x_1, \ldots, x_{i-1}, 1, x_{i+1}, \ldots, x_{n}) \ge f(x_1, \ldots, x_{i-1}, 0, x_{i+1}, \ldots, x_n).
\]
\item (Combinatorial Smallness) For all $i$, 
	$f(\hat{e}_i) = 0$, 
where $\hat{e}_i$ is the $i$-th standard basis vector interpreted as a boolean vector. 
\item (Combinatorial Ampleness)
If $f(\vec{x}) = 0$, then $f(\nx) = 1$.
\end{enumerate}
}

\begin{proof}
We prove the equivalence of each pair of criteria independently.

$\mathbf{(i)} \iff \mathbf{(i')}$: If there exists a nonnegative realization, $(\vec{w}, \theta)$ for $f$, then we have
\begin{align*}
	f(x_1, \ldots, 1, \ldots, x_n) &= \operatorname{sgn}\left( \sum_{j =0, \; j \neq i}^n w_j x_j + w_i - \theta \right)\\
    f(x_1, \ldots, 0, \ldots, x_n) &= \operatorname{sgn}\left( \sum_{j = 0, \; j \neq i}^n w_j x_j - \theta \right).
\end{align*}
Since the argument of the top line is greater than or equal to the argument of the bottom line, we have the desired inequality. 

Conversely, assume that a linear threshold function $f$ is increasing in each argument. Every LTF has some realization (not necessarily positive). Let $(\vec{w}, \theta)$ be a realization with at least one $i$ such that $w_i <0$, and let $\vec{w}'$ be the weight vector which agrees with $\vec{w}$ on all arguments except that $w_i' = 0$. 

We claim that $(\vec{w}' , \theta)$ is also a realization for $f$. Let $f'$ be realized by $(\w', \theta)$, and assume to the contrary that $f' \neq f$. Since for all $i$, we have that $\w'_i \geq \w_i$, we have $f' \geq f$, where $f' \geq f$ means that for all $\x \in  \{0, 1\}^n$, $f'(\x) \geq f(\x)$. Consider some $\vec{x}$ where $f'(\vec{x}) = 1$ and $f(\vec{x}) = 0$. Then $x_i = 1$, because otherwise the functions would equate. Since $w_i' = 0$, \[f'(\vec{x} - \hat{e}_i) = f'(\vec{x}).\] However, since the $i$th component is the only component in which $f$ and $f'$ differ, \[f'(\vec{x} - \hat{e}_i) = f(\vec{x} - \hat{e}_i).\] Substitution thus yields $f(\vec{x}-\hat{e}_i) = 1$ while $f(\vec{x}) = 0$. This violates the positivity of $f$. 

Thus there exists a realization $(\vec{w}, \theta)$ for $f$ with $w_i \ge 0$ for all $i$. This in turn implies the existence of a realization with $w_i > 0$ for all $i$: set the zero entries to a sufficiently small positive $\epsilon$. These new entries do not affect the function value for any $\vec{x} \in \{0, 1\}^n$.\\

\noindent $\mathbf{(ii} \iff \mathbf{(ii')}$: For the first direction, assume $f$ has a small realization $\real$. Then by definition, for all $i$, 
\[
	f(\hat{e}_i) = \operatorname{sgn}(w_i - \theta) = 0. 
\]

Conversely, assume that all realizations of $f$ have $w_i > \theta$ for some $i$. Then we have $f(\hat{e}_i) = \operatorname{sgn}(w_i -\theta) = 1$, for that $i$, so $f$ does not satisfy the combinatorial criterion

\noindent $\mathbf{(iii)} \iff \mathbf{(iii')}$: The first direction is proved by contrapositive: if there exists a complementary pair $\x, \nx$ which satisfy $f(\x) = 0 = f(\nx)$, then any realization $\real$ satisfies \[
	\x \cdot \w - \theta \le 0 \quad \text{and} \quad \nx \cdot \w - \theta \le 0.
\]
These are complementary sums of the $w_i$, so the sum of the two inequalities gives $\sum_{i = 1}^n w_i - 2\theta \le 0$. Thus $f$ cannot have an ample realization and is therefore not ample.

For the opposite direction, let $f$ satisfy the combinatorial statement of ampleness, and let $(\w, \theta)$ be a realization for $f$. Repeating the ``amplification" process from Proposition \ref{amplification} creates a realization $(\w, \theta')$ for $f$ which, by the combinatorial statement, must be ample.

.
\end{proof}
\end{theorem}

\begin{remark}Given the above theorem, we will now omit the term ``combinatorial" when discussing the properties of positivity, ampleness, and smallness. \end{remark}

\begin{remark}
This is naturally a much more testable definition of Semi-Goldilocks and Goldilocks: given a linear threshold function, it suffices to check positivity, that the basis vectors give $0$, and that no negation pairs are both false (for Goldilocks). 
\end{remark}

\subsection{Chow parameters, duality and the Goldilocks criteria}\label{sec:chowgold}

Throughout this section, let $f$ be an LTF with Chow parameters $(m, \vec{a})$.

Chow realized that positivity is easily visible in the parameters. As a corollary of Theorem \ref{chowclassify}, we have the following lemma.

\begin{lemma}\label{chow.pos.lem}
An LTF $f$ with Chow parameters $(m, \vec{a})$ is positive if and only if $a_i \ge 2^{n-1}$ for all $i$.
\end{lemma}

Ampleness is also easily visible in the Chow parameters.

\begin{lemma}\label{chow.ample.lem}
A function is ample if and only if $m \ge 2^{n-1}$.
\begin{proof}
By Theorem \ref{equivGold}, $f$ is ample if and only if $f(\vec{x}) = 0$ implies that $f(\nx) = 1$. If this negation property holds, then $f$ must be true on at least half of the entries, and thus $m = |T| \ge |F|$. Since $|T| + |F| = 2^n$, we have $m \ge 2^{n-1}$. 

Conversely, if there exists a set of negation pairs $\x, \nx$ which are both in $F$, we claim there is no negation pair $\vec{y}, \n{\vec{y}}$ such that both are true. If there was such a pair, then $f$ would violate the asummability property and thus would not be a linear threshold function. Thus every negation pair has at least one false entry, and $\x, \nx$ both evaluate to false, so $m < 2^{n-1}$.
\end{proof}
\end{lemma}

We use the following results on duality and the Goldilocks criteria.
\begin{lemma}\label{dual.pos.lem}
The dual $f^d$ of a positive LTF $f$ is positive.
\begin{proof}
Assume for the sake of contradiction that $f$ is positive, but $f^d$ is not. This implies there exists some vectors $\vec{x}$ and $\vec{y}$, where $\vec{x}, \vec{y}, \vec{x+y} \in \{0, 1\}^n$, such that $f^d(\vec{x}) = 1$, but $f^d(\vec{x} + \vec{y}) = 0$.

Consider $f(\nx)$ and $f(\n{\vec{x+y}})$. We consider the following cases: 

\begin{itemize}
\item If $f^d(\vec{x}) = 1$, this implies that $f(\n{\vec{x}}) = 0$. 

\item If $f^d(\vec{x} + \vec{y}) = 0$, this implies that $f(\n{\vec{x} + \vec{y}}) = 1$. 
\end{itemize}
Note that in all cases we have that $f(\n{\vec{x}}) = 0$ and $f(\n{\vec{x} + \vec{y}}) = 1$. 
Since $\n{\vec{x}} > \n{\vec{x} + \vec{y}}$, this violates our assumption that $f$ is positive, therefore yielding a contradiction. Thus, the dual of a positive LTF must be positive.
\end{proof}
\end{lemma}

\begin{lemma}
An LTF $f$ is ample if and only if $f \ge f^d$.
\begin{proof}
Consider an LTF $f$ and a negation pair $\x, \nx$. For the first direction, assume that $f$ is ample. Thus without loss of generality there are two cases: 
\begin{itemize}
\item Suppose $f(\x) = 1 = f(\nx)$. In this case, $f^d(\x) = 0 = f^d(\nx)$.
\item Suppose $f(\x) = 0$ and $f(\nx) = 1$. In this case, $f^d(\x) = 0$ and $f^d(\nx) = 1$.
\end{itemize}
In both cases, $f(\x) \ge f^d(\x)$ and $f(\nx) \ge f^d(\nx)$ for an arbitrary negation pair in the domain. Thus $f \ge f^d$.

Conversely, if $f$ is not ample, then there exists some $\x, \nx \in \{0, 1\}^n$ such that $f(\x) = 0 = f(\nx)$. On this pair, $f^d(\x) = 1 = f^d(\nx)$, so $f \ngeq f^d$. Either $f < f^d$, or there exists some other vector $\vec{y} \in \{0, 1\}^n$ with $f(\vec{y}) > f^d(\vec{y})$. But in this case, we must have that $f(\n{\vec{y}}) = 1 = f(\vec{y})$, in violation of the asummability criterion.
\end{proof}
\end{lemma}

\begin{corollary}
If $f = f^d$ (self-duality), then $f$ is ample. Otherwise, exactly one of $f, f^d$ is ample.
\end{corollary}

We present some complex interactions between the Goldilocks properties that will be needed later.

\begin{proposition}\label{ps.pas.prop}
Let $f$ be a positive, ample LTF. If $f$ is not self-dual, then the dual of $f$ is small.
\begin{proof}
Assume that $f^d$ is not small. Since $f$ is ample, either $f = f^d$ or $f^d$ is not ample. In the former case the proposition holds, so let $f^d$ be neither small nor ample. 

Let $\hat{e}_i$ be a basis vector for which $f^d(\hat{e}_i) = 1$. Since $f^d$ is not ample, there is some negation pair $\x, \nx \in \{0, 1\}^n$ for which $f^d(\x) = 0 = f(\nx)$. Without loss of generality, let $x_i = 0$ and $\n{x}_i = 1$. Since $f$ is positive and $\n{x}_i \ge \hat{e}_i$, it follows that $f(\nx) = 1$, a contradiction.
\end{proof}
\end{proposition}

\begin{proposition}\label{nsmall.namp.n.prop}
There are exactly $n$ positive LTFs which are self-dual and not small. 
\begin{proof}
Let $f$ be such an LTF, and $\hat{e}_i$ be a singleton for which $f(\hat{e}_i) = 1$. Since $f$ is self-dual, $f(\n{\hat{e}_i}) = 0$. Fix any negation pair $\x, \nx$. One of these, $\x$ without loss of generality, satisfies $x_i = 1$, and thus satisfies $\x \ge \hat{e}_i$. Thus $f(\x) = 1$ and $f(\nx) = 0$. In particular, $f(\hat{e}_j) = 0$ for all $j \neq i$. Thus the choice of which $i$ is large completely identifies $f$, and there are exactly $n$ such choices.
\end{proof}
\end{proposition}

\section{An important counting lemma}
\label{sec:nondeg}
This section depends heavily upon the theory developed in Section \ref{sec:properties}. The goal of this section is to give a formula relating the number of LTFs on $n$ variables to the number of Semi-Goldilocks LTFs (see Lemma \ref{ltf.nondegen.formula.lem}). This lemma will be key in obtaining our asymptotic results in Section \ref{sec.asymptotics}. We begin with a bijective correspondence between Semi-Goldilocks LTFs and positive LTFs which are \emph{nondegenerate}.

\begin{definition}\label{def:degnondegen}
Given an LTF $f$ on $n$ variables, define the \emph{degree} $\deg(f)$ of $f$ to be the number of variables for $f$ which are not weak (see Theorem \ref{chowclassify}). If $\deg(f) = n$, then $f$ is said to be \emph{nondegenerate}. 
\end{definition}

\begin{remark} 
Nondegenerate LTFs, in the literature, are sometimes referred to as LTFs on ``exactly $n$ variables.'' Degenerate LTFs are viewed as being fundamentally on $\deg(f)$ variables. To prevent confusion with the actual number of arguments $n$, we use Winder's terminology of degeneracy.
\end{remark}

First, let $\calS$ be the set of Semi-Goldilocks LTFs, and let $\calT$ be the set of nondegenerate positive LTFs. Let $f$ be the linear threshold function realized by $(\w, 1)$, and let $(m, \vec{a})$ be the Chow parameters of $f$. Define $\w'$ as the weight where $w_i' = 2$ if the Chow parameter $a_i$ for $f$ is $2^{n-1}$ and where $w'_i = w_i$ otherwise. 

Let $\phi: \calS \rightarrow \calT$ be a map such that $\phi(f) = h$, where $h$ is a linear threshold function realized by $\w'$. We prove the following statement regarding $\phi$.

\begin{lemma}
\label{thm.phiworkswell}
The map $\phi : \calS \rightarrow \calT$ from Semi-Goldilocks LTFs to nondegenerate positive LTFs is well defined, i.e. $\phi$ does not depend on our choice of realization $\w$, and $\phi(f) \in \calT$.

\begin{proof}
First, we consider two realizations $\w$ and $\w'$ for the same $f$, and their images $\phi(\w)$ and $\phi(\w')$. We claim that these images are realizations of the same boolean function. Since $\w$ and $\w'$ were weak in the same variables, $\phi(\w)(\x) = 1 = \phi(\w')(\x)$ whenever $x_i = 1$ for any of these formerly weak variables. For all other $\x$, the values of $\phi(\w)(\x)$ and $\phi(\w')(\x)$ are determined by unchanged summations of terms in $\w$ and $\w'$. Since $\w$ and $\w'$ are both realizations of $f$,
\[
	\phi(\w)(\x) = f(\x) = \phi(\w')(\x).
\]

To check that $\phi(f) \in \calT$ it suffices to show that $\phi$ takes a Semi-Goldilocks LTF to a non-degenerate positive one, i.e. that $\deg(f) = n$. This is true precisely when $f$ has no \emph{weak} variables. We note that this holds by construction, since any Chow parameter for $f$ which is weak (i.e. is $2^{n-1}$) is replaced by one which is not weak.
\end{proof}
\end{lemma}

\begin{theorem}
\label{thm.semigoldbij}
The map $\phi : \calS \rightarrow \calT$ (that is, from Semi-Goldilocks LTFs to nondegenerate positive LTFs) is a bijection. 
\begin{proof}
First, we prove that $\phi$ is an injective map. Suppose for two Semi-Goldilocks LTFs $f_1$ and $f_2$ that we have $\phi(f_1) = \phi(f_2) = h$. Note that $h$, by definition of $\phi$, is a nondegenerate positive LTF can be realized by some weight $\w'$. Therefore, $\w'$ contains only positive weights and has Chow parameters which are all greater than $2^{n-1}$. For each $i$, either $w'_i \in (0, 1]$ or $w'_i > 1$. Given $\w'$, we can generate a new weight $\w$ as follows: if $w'_i > 1$, then $w_i = \epsilon$ for a sufficiently small value of $\epsilon$, and $w_i = w'_i$ otherwise. Since $\w$ can realize both $f_1$ and $f_2$, we have that $f_1 = f_2$. 

Now, we prove that $\phi$ is a surjective map as well. Any nondegenerate positive linear threshold function $h$ can be realized by some weight $\w'$, where $\w'$ yields a positive LTF. Note that for all $i$, either $w'_i \in (0, 1]$ or $w'_i > 1$. We can use $\w'$ to obtain a new weight $\w$ (if $w'_i > 1$, then $w_i = \epsilon$ for a sufficiently small value of $\epsilon$, and $w_i = w'_i$ otherwise) that is both positive and small. The weight $\w$ can realize a Semi-Goldilocks function $f$ such that $\phi(f) = h$. Therefore, for all $h \in \calT$, there exists $f \in \calS$ such that $\phi(f) = h$. 
\end{proof}
\end{theorem}

\subsection{Weak weights and degree}

The following classification of weak variables is useful.

\begin{proposition}
Given a linear threshold function $f$ and a variable $i$, the LTF $f$ is weak in $i$ if and only if
\[
	f(x_1, \ldots, x_{i-1}, 1, x_{i+1}, \ldots, x_n) = f(x_1, \ldots, x_{i-1}, 0, x_{i+1}, \ldots, x_n)
\]
for all elements of the domain.
\begin{proof}
For the first direction, if this equality holds on the domain, then for each pair of vectors $\x, \nx^i$ which differ only in the $i$-th component, exactly one of the two vectors is counted towards the sum $|T_{x_i = 1}| + |F_{x_i = 0}|$. Thus, this sum is exactly half the domain, and $a_i = 2^{n-1}$. Thus $f$ is weak.

By definition, $f$ has Chow parameter $a_i = 2^{n-1}$. Thus, 
\[
	2^{n-1} = |T_{x_i = 1}| + |F_{x_i = 0}|.
\]
Consider a pair of vectors $\x, \nx^i \in \{0, 1\}^n$ which differ only in the value of the $i$-th component. The entire domain can be split into $2^{n-1}$ such pairs. Assume for the sake of contradiction that $f(\x) = 1$ and $f(\nx^i) = 0$. Then either both or neither of the vectors are counted towards the value of $a_i$. In order to preserve the equality with half of the domain, there must be some other pair which differ only in $i$, denoted $\y, \ny^i$, such that $f(\y) = 0$ and $f(\ny^i) = 1$ but that $x_i \neq y_i$. The pair $\y, \ny^i$ are then neither (resp. both) counted towards the value of $a_i$, preserving the equality. 

Now consider the pairs $(\x, \ny^i)$ and $(\y, \nx^i)$. The value of $f$ is 1 on the first pair and 0 on the second pair, but $\x + \ny^i = \y + \nx^i$. Thus, we have violated the asummability criterion of Corollary \ref{simpbad.cor}, and $f$ cannot be a linear threshold function. Thus $f(\x) = f(\nx^i)$ for all pairs of vectors which differ only in the $i$-th component.
\end{proof}
\end{proposition}

The next two propositions allow us to formulate the number of LTFs in terms of nondegeneracy (Lemma \ref{ltf.nondegen.formula.lem}).

\begin{proposition}\label{nchoosekreps.prop}
There is an ${n \choose k}$-to-1 map $T_{n, k}$ from LTFs on $n$ variables with degree $k$ to nondegenerate LTFs on $k$ variables.
\begin{proof}
Recall that by definition an LTF of degree $k$ has $n-k$ weak weights. By the above proposition, for each weak weight the function is identical whether that value is chosen to be 0 or 1. Thus we have
\[
	f(\x) = f'(x_1, \ldots, x_{i-1}, x_{i+1}, \ldots, x_n) 
\]
on the whole domain. More generally, $f$ can be identified with a function on $k$ variables which completely identifies it at all points on the domain by removing the weak weights. However, any two functions which differ only in their placement of the $n-k$ weak weights will map to the same reduced function under this identification, and so there are ${n \choose n-k} = {n \choose k}$ elements of each fiber of the map.
\end{proof}
\end{proposition}

\begin{proposition}\label{nclassnondegen.prop}
The N-class (orbit under $\vu$-complementation) of a nondegenerate LTF on $n$ variables has exactly $2^n$ distinct LTFs.
\begin{proof}
Since $f$ is nondegenerate, for each $i$ there exist two elements of the domain $\x, \nx^i$, differing only in the $i$-th place, for which $f(\x) \neq f(\nx^i)$, and there exists such a pair for every $i$. In order for a $\vu$-complementation to fix all the elements of the domain, it must be trivial. Thus, each of the $2^{n}$ distinct nontrivial $\vu$-complementations takes $f$ to some other nondegenerate LTF. Since each of these functions will have a different set of Chow parameters (by Proposition \ref{chow.perm.prop}) they are all distinct. Thus there are $2^n$ distinct LTFs in the N-class of $f$.
\end{proof}
\end{proposition}

\subsection{Formula for LTFs in terms of $\SGold(k)$}

We prove the following exact formula for $\LTF(n)$ in terms of $\SGold(k)$ where $k \le n$.
\begin{lemma}\label{ltf.nondegen.formula.lem}
For all $n$, the number of linear threshold functions on $n$ variables is given by
\[
	\LTF(n) = \sum_{k=0}^n {n \choose k} 2^k \SGold(k).
\] \end{lemma}

\begin{proof}
First, we note that linear threshold functions can be partitioned according to their degree. Then the identification $T_{n, k}$ (Proposition \ref{nchoosekreps.prop}) implies:
\[
	\LTF(n) = \sum_{k = 0}^n {n \choose k} |\{f \in \Thresh_k : \text{$f$ nondegenerate}\}|.
\]
By Proposition \ref{nclassnondegen.prop}, these nondegenerate LTFs can be split into N-type classes, each with $2^k$ elements, of which exactly one is positive. Thus we have
\[
	\LTF(n) = \sum_{k=0}^n {n \choose k} 2^k |\{f \in \Thresh_k : \text{$f$ positive and nondegenerate}\}|
\]
which becomes by Theorem \ref{thm.semigoldbij}
\[
\LTF(n) =\sum_{k=0}^n {n \choose k} 2^k \SGold(k). 
\]
\end{proof}

This lemma will be key in the next section.

\section{Asymptotic results}
\label{sec.asymptotics}
Our starting point is the asymptotic formula for $\LTF(n)$, the number of linear threshold functions on $n$ variables, due to Irmatov.

\begin{theorem}\cite{irmatov96}
\label{ltf.asymp.irmatov.thm}
The number of linear threshold functions on $n$ variables satisfies the asymptotic formula
\[
	\LTF(n) \sim 2 \sum_{i = 0}^n {2^n - 1 \choose i} \sim 2^{n^2 - n \log_2{n} + \gamma(n)}, 
\]
where $\gamma(n)$ is some monotonic function satisfying $\gamma(n) = O(n)$.
\end{theorem}

This implies the following asymptotic result, which is frequently useful.

\begin{lemma}\label{ltf.ratios.lem}
The following asymptotic ratio holds: 
\[
	\frac{\LTF(n-1)}{\LTF(n)} \sim \frac{4(n-1)}{2^{2n + \gamma(n)-\gamma(n-1)}}. 
\]
\begin{proof}
Applying Theorem \ref{ltf.asymp.irmatov.thm} gives
\[
	\frac{\LTF(n-1)}{\LTF(n)} \sim \frac{2^{(n-1)^2 - (n-1) \log_2{(n-1)} + \gamma(n-1)}}{2^{n^2 - n \log_2{n} + \gamma(n)}}.
\] 
This exponent can be simplified to give 
\[
	\log_2 \left( \frac{\LTF(n-1)}{\LTF(n)} \right) \sim -2n + 1 + n \log_2 n - (n-1) \log_2 (n-1) + \gamma(n-1) - \gamma(n).
\]
In the limit as $n \to \infty$, we have
\[
	n \log_2 n - (n-1) \log_2 (n-1) = \log_2 \left( \left(\frac{n}{n-1} \right)^n \right) + \log_2 (n-1) \to (1 + \log_2(n-1)).
\]
Thus the exponent becomes
\[
	\log_2 \left( \frac{\LTF(n-1)}{\LTF(n)} \right) \sim -2n + 2 + \log_2(n-1)  + \gamma(n-1) - \gamma(n).
\]
Thus we have
\[
	\frac{\LTF(n-1)}{\LTF(n)} = \frac{4(n-1)}{2^{2n + \gamma(n) - \gamma(n-1)}}.
\]
\end{proof}
\end{lemma}

\begin{remark}
Note that since $\gamma(n) = O(n)$, the difference $\gamma(n) - \gamma(n-1)$ is bounded above and below by constants. This formula implies
\[
	\frac{\LTF(n-1)}{\LTF(n)} = \Theta\left(\frac{n}{2^{2n}}\right). 
\]
\end{remark}

Using Lemma \ref{ltf.ratios.lem}, we prove asymptotic results first on the number of Semi-Goldilocks functions $\SGold(n)$ and then on the number of Goldilocks functions $\Gold(n)$.

\subsection{An asymptotic formula for $\SGold(n)$}
In Section \ref{sec:nondeg}, we proved the following exact formula for $\LTF(n)$ in terms of $\SGold(k)$ where $k \le n$ (see Lemma \ref{ltf.nondegen.formula.lem}).

\begin{lemma}[see Lemma 5.7] For all $n$, the number of linear threshold functions on $n$ variables is given by
\[
	\LTF(n) = \sum_{k=0}^n {n \choose k} 2^k \SGold(k).
\] \end{lemma}

Lemma \ref{ltf.nondegen.formula.lem} implies that
\begin{equation}\label{psn.eqn}
	\SGold(n) = \frac{\LTF(n) - \LTF(n-1)}{2^n} - \frac{1}{2^n} \sum_{k = 0}^{n-1} {n-1 \choose k-1} 2^k \SGold(k).
\end{equation}
We view this second term as an error term $\epsilon(n)$, so that
\[
	\epsilon(n) := \frac{1}{2^n} \sum_{k = 0}^{n-1} {n-1 \choose k-1} 2^k \SGold(k).
\]
We prove that $\frac{\epsilon(n)}{\SGold(n)} \to 0$ as $n \to \infty$, and thus derive an asymptotic formula.
\begin{theorem}\label{ps.asymp.thm}
As $n \to \infty$, 
\[
	\SGold(n) \sim \frac{\LTF(n) - \LTF(n-1)}{2^n}. 
\]
\end{theorem}

In order to prove this, we find a bound on the error term.
\begin{lemma}\label{epsilonsmall.lem}
The error term satisfies $\epsilon(n) \le \LTF(n-1)$ for all $n$.
\begin{proof}
Consider the difference $\LTF(n-1) - \epsilon(n)$. Substituting the definition of $\epsilon$ and using Lemma \ref{ltf.nondegen.formula.lem} gives 
\begin{equation}\label{diff.eqn}
	\LTF(n-1) - \epsilon(n) = \sum_{k = 0}^{n-1} \left({n-1 \choose k} - \frac{1}{2^n} {n-1 \choose k-1}\right) 2^k \SGold(k).
\end{equation}

It is elementary to check by ratios that ${n-1 \choose k} \ge \frac{1}{2^n} {n-1 \choose k-1}$ for all $n$ and all $k \in \{0, \ldots, n-1\}$. Thus each term of the summation in (\ref{diff.eqn}) is nonnegative, and $\LTF(n-1) \ge \epsilon(n)$.
\end{proof}
\end{lemma}

\begin{proof}[Proof of Theorem \ref{ps.asymp.thm}]
We consider the ratio $\frac{\epsilon(n)}{\SGold(n)}$:
\[
	\frac{\epsilon(n)}{\SGold(n)} = \frac{1}{2^n \SGold(n)} \sum_{k = 0}^{n-1} {n-1 \choose k-1} 2^k \SGold(k) = \frac{1}{2^n \SGold(n)} \sum_{k = 0}^{n-1} {n-1 \choose k} \frac{k}{n-k} 2^k \SGold(k).
\]
Note that $\frac{k}{n-k} \le n$ for all $k \in \{0, \ldots,  n-1\}$. Thus this equation implies an upper bound on the error given by
\begin{equation}\label{errorbound.eqn}
	\frac{\epsilon(n)}{\SGold(n)} \le \frac{n}{2^n \SGold(n)}\sum_{k = 0}^{n-1} {n-1 \choose k-1} 2^k \SGold(k) = \frac{n \LTF(n-1)}{2^n \SGold(n)}.
\end{equation}

We show now that this ratio vanishes. From the equation (\ref{psn.eqn}), we derive
\[
	2^n \SGold(n) = \LTF(n) \left( 1 - \frac{\LTF(n-1) + \epsilon(n)}{\LTF(n)}\right) \ge \LTF(n) \left( 1 - 2\frac{\LTF(n-1)}{\LTF(n)}\right).
\]
By Lemma \ref{ltf.ratios.lem}, the term inside the parentheses approaches $1$. Thus, for any constant $c < 1$, there exists some $N$ such that the term is greater than $c$ for all $n \ge N$. For $c = \frac{1}{2}$, the bound becomes 
\[
	\frac{n\LTF(n)}{2^n \SGold(n)} \le \frac{n}{2}.
\]
From the equation (\ref{errorbound.eqn}), multiplying by $\frac{\LTF(n-1)}{\LTF(n)}$ gives
\[
	\frac{\epsilon(n)}{\SGold(n)} \le \frac{n \LTF(n-1)}{2^n \SGold(n)} \le \frac{n}{2} \left( \frac{\LTF(n-1)}{\LTF(n)}\right)
\]
which also vanishes by the formula in Lemma \ref{ltf.ratios.lem}. Thus the error term vanishes, and we have
\[
	\SGold(n) \sim \frac{\LTF(n) - \LTF(n-1)}{2^n}.
\]
\end{proof}

\begin{corollary}\label{ps.simp.asymp.cor}
An asymptotic formula for the number of chambers in the chamber decomposition of $\Weightsg$ is
\[
	\SGold(n) \sim \frac{\LTF(n)}{2^n} \sim 2^{n^2 - n\log_2 n + O(n)}.
\]
\end{corollary}
\begin{proof} By Theorem \ref{ps.asymp.thm},  we have
\[
	\frac{2^n \SGold(n)}{\LTF(n)} \to 1,
\] as $n \to \infty$. \end{proof}

\subsection{An asymptotic formula for $\Gold(n)$}

Let the number of positive LTFs be denoted $\Pos(n)$. We prove that the number of positive and small LTFs is asymptotically equal to the total number of positive LTFs.

\begin{lemma}\label{pos.asymp.mostsmall}
As $n \to \infty$, 
$$\lim_{n \to \infty} \frac{\Pos(n)}{\SGold(n)} = 1.$$
\begin{proof}
Note that we have the following elementary relationship from Proposition \ref{nchoosekreps.prop}:
$$\Pos(n) = \sum_{k=0}^n |\{f \text{ positive} \; | \; \deg(f) = k\}| = \sum_{k=0}^n {n \choose k} \SGold(k) = \SGold(n) + \sum_{k=0}^{n-1} {n \choose k} \SGold(k).$$
This implies a set of bounds for the ratio given by
$$1 \le \frac{\Pos(n)}{\SGold(n)} \le 1 + \frac{\SGold(n-1)}{\SGold(n)} \sum_{k = 0}^{n-1} {n \choose k} \le 1 + 2^n \frac{\SGold(n-1)}{\SGold(n)}. $$
Applying the asymptotics of Corollary \ref{ps.simp.asymp.cor} and Lemma \ref{ltf.ratios.lem} yields
$$ 2^n \frac{\SGold(n-1)}{\SGold(n)} \le \frac{2^{n+3}(n-1)}{2^{2n}} = \frac{8(n-1)}{2^n} \to 0,$$
and thus $\lim_{n \to \infty} \frac{\Pos(n)}{\SGold(n)} = 1$.
\end{proof}
\end{lemma}

Using this result, we can prove an asymptotic formula for $\Gold(n)$.

\begin{theorem}\label{asymp.thm}
We have the following asymptotic formula for $\Gold(n)$ in terms of $\LTF(n)$: 
\[
	\Gold(n) \sim \frac{\LTF(n)-\LTF(n-1)}{2^{n+1}}. 
\]

In order to prove this asymptotic, we use the following lemma. 

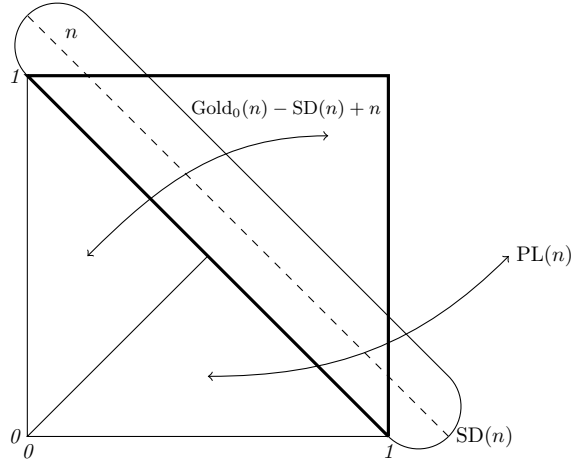
\begin{figure}
\begin{center}
\begin{tikzpicture}[scale=0.8, every node/.style={scale=0.8}]

\draw (0,0) -- (0,6) -- (6,6) -- (6,0) -- (0,0) -- (3,3);
\draw [dashed] (7,0) -- (0,7);
\draw (7,0) to[in=-45,out=-135] (6,0) to (0,6) to[in=-135,out=135] (0,7);
\draw (7,0) to[in=-45,out=45] (7,1) to (1,7) to[in=45,out=135] (0,7);
\draw [<->] (3,1) to[in=-135,out=0] (8,3);
\draw [<->] (1,3) to[in=180,out=45] (5,5);

\draw [very thick] (6,0) -- (6,6) -- (0,6) -- (6,0);

\node [right] at (7,0) {$\SD(n)$};
\node [above] at (4.3,5.15) {\small $\Gold(n) - \SD(n) + n$};
\node [below] at (0,0) {0};
\node [below] at (6,0) {1};
\node [left] at (0,0) {0};
\node [left] at (0,6) {1};
\node [right] at (8,3) {$\PL(n)$};
\node [above right] at (0.5,6.5) {$n$};

\end{tikzpicture}

\caption{A partition of positive LTFs. The square represents the hypercube of positive and small LTFs, the bolded triangle represents the Goldilocks LTFs, and the oval represents all self-dual threshold functions. Arrows indicate equal cardinality of regions due to the bijection between dual functions.}
\label{cube.diag}
\end{center}
\end{figure}

\begin{lemma}\label{gold.asymp.lem}
As $n \to \infty$, the ratio of $\Gold(n)$ to $\SGold(n)$ approaches $\frac12$, that is, 
$$
	\lim_{n \to \infty} \frac{\Gold(n)}{\SGold(n)} = \frac{1}{2}.
$$

\begin{proof}
For one direction, we note that there is a natural partitioning of the space of positive LTFs according to the Goldilocks criteria (Figure \ref{cube.diag}). The dual map bijects those LTFs which are not ample with those which are ample and not self-dual, implying a counting equation:
\begin{align}\label{diagram.count.posn.eqn}
\Pos(n) &= 2(\Gold(n) - \SD(n)+n) + 2\PL(n) + \SD(n) \nonumber \\ \SGold(n) + \PL(n) + n &= 2\Gold(n) + 2\PL(n) - \SD(n) + 2n
\end{align}
where $\PL(n)$ is the number of positive LTFs which are not self-dual and not small. Elements of $\PL(n)$ must be ample, since they have some negation pair $\hat{e}_i, \n{\hat{e}_i}$ such that $f$ is true on both. Were they not ample, the pairs $(\hat{e}_i, \n{\hat{e}_i})$ and $(\x, \nx)$ with $f(\x) = 0 = f(\nx)$ would violate the asummability criterion. The term $n$ is necessary to account for the exactly $n$ positive LTFs which are self-dual but not small (Proposition \ref{nsmall.namp.n.prop}).

Manipulation of Equation \ref{diagram.count.posn.eqn} gives
$$\frac{2\Gold(n) + \PL(n) - \SD(n) + n}{\SGold(n)} = 1.$$
Lemma \ref{pos.asymp.mostsmall} implies that $\frac{\PL(n)}{\SGold(n)} \to 0$. Furthermore, since the self-dual LTFs on $n$ variables biject with all LTFs on $n-1$ variables (Theorem \ref{sd.bijection}), we have that $\SD(n) = \LTF(n-1)$. The asymptotics for $\SGold(n)$ and Corollary \ref{ps.simp.asymp.cor} give:
\begin{align*}\lim_{n \to \infty} \frac{\LTF(n-1)}{\SGold(n)} &= \lim_{n \to \infty} \left(\frac{\LTF(n-1)}{\LTF(n)} \right) \left(\frac{\LTF(n)}{\SGold(n)}\right)\\ &\le \lim_{n \to \infty} \frac{2^{n+2}(n-1)}{2^{2n}} = 0.
\end{align*}

Combining these results gives
$$
1 = \lim_{n \to \infty} \left( \frac{2\Gold(n) + \PL(n) - \SD(n) + n}{\SGold(n)} \right) = \lim_{n \to \infty} \frac{2\Gold(n)}{\SGold(n)}. 
$$
\end{proof}
\end{lemma}

\begin{proof}[Proof of Theorem \ref{asymp.thm}]
Using our asymptotic for $\SGold(n)$ \ref{ps.asymp.thm}, we have
\[
\Gold(n) \sim \frac{\SGold(n)}{2} \sim \frac{\LTF(n)-\LTF(n-1)}{2^{n+1}}. 
\]
\end{proof}
\end{theorem}

\begin{corollary}\label{!gold.asymp}
The number of chambers in the chamber decomposition of $\Weights$ satisfies the following asymptotic formula:
$$\Gold(n) \sim \frac{\LTF(n)}{2^{n+1}} \sim 2^{n^2 - n\log_2{n} + O(n)}.$$
\end{corollary}

\begin{remark}
While the asymptotics for $\SGold(n)$ and $\Gold(n)$ may appear identical, they differ by a factor of $2$, which is included in the $O(n)$ of the exponent. 
\end{remark}

\section{Testing for the Goldilocks criteria is hard}
\label{sec.hardness}
In order to understand the computational resources necessary to enumerate Semi-Goldilocks or Goldilocks functions, we turn to complexity theory. 

\begin{definition} 
\label{conpcomplete}
A decision problem is \emph{co-NP} if, given that the answer to the problem is no, there exists a certificate, i.e. a proof of the answer with size polynomial to the size of the input, that can be verified by a deterministic algorithm in polynomial time. A decision problem is \emph{co-NP-hard} if any co-NP problem can be reduced to it via a polynomial-time algorithm. A decision problem is \emph{co-NP-complete} if it is both in co-NP and co-NP-hard. 
\end{definition}

Intuitively, co-NP-hard problems are at least as hard as any problem in co-NP. \\

The problem of deciding whether a boolean function $f$ is a threshold function, denoted \Thresf, is known to be co-NP-complete \cite{hegedusconp}. We prove the analogous results for our subclasses of LTFs.

\begin{theorem}\label{coNPcomp.sgold.thm}
Given an arbitrary boolean function $f$, the problem of determining whether $f$ is a Semi-Goldilocks function, \SGoldf, is co-NP-complete.
\begin{proof}
First, we show that \SGoldf is in co-NP. Assume that $f$ is not a Semi-Goldilocks function. There are three (possibly overlapping) cases, each with a witness, i.e. a certificate that $f$ is not Semi-Goldilocks, which can be checked in polynomial time:
\begin{enumerate}[label=(\roman*)]
\item $f$ is not an LTF. In this case, there is some sequence of $c_i$ violating the asummability criterion (Theorem \ref{asum.thm}). This sequence serves as a witness for the nonseparability of $f$ -- checkable in polynomial time by evaluating the corresponding sums.
\item $f$ is not positive. In this case, a witness is a pair $\x, \y \in \{0, 1\}^n$ such that $\x \le \y$ but $f(\x) > f(\y)$.
\item $f$ is not small. In this case, a witness is a value $i$ for which $f(\hat{e}_i) = 1$.
\end{enumerate}
Thus, appending each witness with an integer identifying the case allows an algorithm to verify that $f$ is not Semi-Goldilocks in polynomial time. It follows that \SGoldf is in co-NP.

We now show \SGoldf is co-NP-hard by showing that solving an instance of \Thresf reduces to solving an instance of \SGoldf. Let $f$ be some arbitrary boolean function on $n$ variables, such that the input size is $N = 2^n$. We perform the following reductions. 

\begin{enumerate}[label=(\roman*)] 
\item Given an arbitrary $f$,  we first compute the Chow parameters of $f$ to determine in which variables $f$ is increasing or decreasing. Then, we negate the variables on which $f$ is decreasing via $\vu$-complementation, for the $\vu$ which satisfies $u_i = 1$ if and only if $f$ is decreasing in $i$. This generates a new boolean function, $\fp$, which is necessarily positive. Computing the Chow parameters of $f$ and constructing $\fp$ takes time $O(N^2)$. 

\item Now, consider our boolean function $\fp$. Generate a new function $\fps$ as follows. For each $\x \in \{0, 1\}^n$, create a new vector $\y$, where 
\[
y_i = 
\begin{cases}
\begin{aligned}
0 \quad &\text{if} \quad \fp(\hat{e}_{i}) = 1 \\ 
x_i \quad &\text{if} \quad \fp(\hat{e}_{i}) = 0.
\end{aligned}
\end{cases}
\]
Now, for all $\x \in \{0, 1\}^n$, we define $\fps(\x) = \fp(\y)$. This reduction $\fps$ is an LTF if and only if $\fp$ is an LTF, with realizations which coincide except for the addition or removal of weights with $w_i > \theta$. Furthermore, $\fps$ is small, because $\fps(\hat{e}_i) = 0$ for all $i$. This process takes time $O(N)$.  

If $\fps$ is a Semi-Goldilocks LTF, since steps (i) and (ii) do not change the LTF status of $f$, then $f$ is an LTF. On the other hand, if $\fps$ is not a Semi-Goldilocks LTF, since we know that $\fps$ is positive and small, this implies that $f$ is not an LTF. 

\end{enumerate}

Therefore, we can reduce the decision problem of \Thresf to \SGoldf in polynomial time. Since \Thresf is co-NP-hard, \SGoldf is co-NP-hard as well. Thus the problem of deciding whether a given function is Semi-Goldilocks is co-NP-complete. 

\end{proof}
\end{theorem}

Using our result on the complexity of \SGoldf, we prove a similar result for \Goldf. 

\begin{theorem}\label{coNPcomp.thm}
Given an arbitrary boolean function $f$, the problem of determining whether $f$ is a Goldilocks function, \Goldf, is co-NP-complete.\end{theorem}

\begin{proof}
First, we show that \Goldf is in co-NP. Assume that $f$ is not a Goldilocks function. There are two possible cases to consider: 
\begin{enumerate}[label=(\roman*)]
\item $f$ is not a Semi-Goldilocks LTF. In this case, there is some witness (as previously described) that can be verified in polynomial time. 
\item $f$ is Semi-Goldilocks and not ample. In this case, a witness is a negation pair $\x, \nx$ on which $f(\x) = 0 = f(\nx)$. 
\end{enumerate}
Each witness can be appended with an integer identification, which allows an algorithm to verify that $f$ is not Goldilocks in polynomial time. It follows that \Goldf is in co-NP.

We now show \Goldf is co-NP-hard by showing that solving an instance of \SGoldf reduces to solving an instance of \Goldf. Let $f$ be some arbitrary boolean function on $n$ variables, such that the input size is $N = 2^n$. 

We perform the following reduction. We calculate the dual of $f$, which we denote $f^d$. We then return 
\begin{equation}\label{return.eqn}
\textproc{Gold}(f) \vee \textproc{Gold}(f^d). 
\end{equation}

The process of calculating the dual and evaluating the above expression takes time $O(N)$. 

If $f$ is a Goldilocks LTF, then this means $f$ is a Semi-Goldilocks LTF as well, so the above expression returns true. On the other hand, if $f$ is not a Goldilocks LTF, one of the two cases below must occur: 

\begin{enumerate}
\item $f$ is a Semi-Goldilocks LTF that is not ample. In this case, $f^d$ must therefore be ample, so $f^d$ is a Goldilocks LTF, so the expression (\ref{return.eqn}) evaluates to true.  
\item $f$ is not a Semi-Goldilocks LTF. In this case, neither $f$ nor $f^d$ is a Goldilocks LTF, so the expression evaluates to false.
\end{enumerate}

Therefore, we can reduce \SGoldf to \Goldf. Since \SGoldf is co-NP-hard, \Goldf is co-NP-hard as well. Thus the problem of deciding whether a given function is a Goldilocks LTF is co-NP-complete. 
\end{proof}

\appendix

\section{An algorithm for enumeration of chambers of $\Weightsgen$}\label{sec.alg}

\begin{algorithm}
\caption{}
\label{algorithm1}
\begin{algorithmic}[1]
\Procedure{SGoldSD}{$f, \vec{a}$} \Comment{Counts semi-Goldilocks functions in SD class}
\State{$s \gets 0$}
\For{$i \in \{0, \ldots, n\}$}\label{alg.sg.ired}
	\If{$a_i \neq a_{i-1}$} \Comment{If distinct from the previous $i$}\label{alg.sg.dist}
        \If{\textproc{isSmall}$(f, i, 1) == 1$} \Comment{If $f_{x_i = 1}$ is small}\label{alg.sg.1small}
            \State{$s \gets s + \textproc{PermCount}(f, i, 1, \vec{a})$}\label{alg.sg.1add}\Comment{Add the number of permutations of $f_{x_i = 1}$}
        \EndIf
        \If{\textproc{isSmall}$(f, i, 0) == 0$ and $a_i \neq 2^{n-1}$}\label{alg.sg.0small}\Comment{If $f_{x_i = 0}$ is small and distinct from $f_{x_i= 1}$}
        	\State{$s \gets s + \textproc{PermCount}(f, i, 0, \vec{a})$}\label{alg.sg.0add} \Comment{Add the number of permutations of $f_{x_i = 0}$}
        \EndIf
    \EndIf
\EndFor\label{alg.sg.endloop}
\State{\Return $s$}
\EndProcedure
\State{}
\Procedure{GoldSD}{$f, \vec{a}$} \Comment{Counts Goldilocks functions in SD class}
\State{$s \gets 0$}
\For{$i \in \{0, \ldots, n\}$}
	\If{$a_i \neq a_{i-1}$} \Comment{If distinct from the previous $i$}
    \State{$\mathrm{val} \gets 0$} \Comment{Which reduction is ample?}
	\If{$b_i > 2^{n-1}$} \Comment{If $f_{x_i = 1}$ is the ample reduction}
    	\State{$\mathrm{val} \gets 1$}
    \EndIf
    \If{\textproc{isSmall}$(f, i, \mathrm{val}) == 1$}\Comment{If the ample assignment is small}
        	\State{$s \gets s + \textproc{PermCount}(f, i, \mathrm{val}, \vec{a})$}\Comment{Add the number of permutations of $f_{x_i = \mathrm{val}}$}
        \EndIf
    \EndIf
\EndFor
\State{\Return{$s$}}
\EndProcedure
\end{algorithmic}
\end{algorithm}

We present algorithms for the enumeration of Semi-Goldilocks and Goldilocks functions, which by Corollary \ref{cor:bijection} enumerate the number of chambers in $\Weightsg$ and in $\Weights$, respectively. We follow primarily the canonical LTF enumeration algorithm due to Winder \cite{sevenvar} and adapted by Muroga et al. \cite{eightvar}. This algorithm generates one element of each self-dualization class (see Definition \ref{sd.class.def}) in the form of a self-dual function on $n+1$ variables. \\

We prove the following lemma due to Winder.

\begin{lemma}[See Lemma \ref{canonical.generate.lem}]
\cite{Winder61}
Given a self-dual representative $f :\{0, 1\}^{n+1} \to \{0, 1\}$ for an SD class, every element of that SD class which is a function on $n$ variables can be reached by the following process:
\begin{enumerate}[label=(\roman*)]
\item setting a single argument of $f$ to either 1 or 0,
\item permuting some remaining arguments of $f$,
\item $\vu$-complementing some remaining arguments of $f$, 
\end{enumerate}
and every LTF so obtained is in the SD class of $f$.
\end{lemma}

On this basis we present a pair of subroutines (Algorithm \ref{algorithm1}) which take as input a positive, self-dual LTF $f$ on $n+1$ variables with Chow parameters $\vec{a} = (2^n, 2a_0, \ldots, 2a_n)$ (see Definition \ref{chow parameters}) in monotonic order and return the number of semi-Goldilocks (\textproc{SGoldSD}) or Goldilocks (\textproc{GoldSD}) functions on $n$ variables in the SD class containing $f$. 

Winder's enumeration algorithm generates a single element of each SD class which contains any LTFs on $n$ variables. In this section we prove the following theorems. 

\begin{theorem}[see Theorems \ref{alg.semivalid.thm} and \ref{alg.valid.thm}]
The sum of \textproc{SGoldSD}$(f, \vec{a})$ (resp. \GoldSD{f, \vec{a}}) across all canonical generators $f$ output by Winder's algorithm is the total number of Semi-Goldilocks (resp. Goldilocks) functions on $n$ variables.
\end{theorem}

\subsection{Variable assignments: a generalization of anti-self-dualization}
In order to apply Lemma \ref{canonical.generate.lem}, we formalize the process of setting an argument to 0 or 1. Let $f$ be a self-dual LTF on $n+1$ variables. If the anti-self-dualization $F$ has Chow parameters $(a_0, \ldots, a_n)$, it follows from Proposition \ref{chow.sdual.prop} that $f$ has parameters $(2^n, 2^{n+1} - 2a_0, \ldots, 2a_n)$.

\begin{definition}
We denote by $f_{x_{i}=0}$ the assignment of the $i$th variable in $f$ to 0; that is,
\[
	f_{x_{i}=0}(x_1, \ldots, x_n) := f(x_1, \ldots, x_{i-1}, 0, x_i, \ldots, x_n).
\]
Similarly, we define $f_{i = 1}$ to be the assignment of the $i$th variable in $f$ to $1$. We call such a function the \emph{false} and \emph{true $i$-reductions}, respectively.

We use corresponding notation for the true and false sets, so that, for example, $T_{x_i = 1}$ is the set of elements $\x$ on the domain with $x_i = 1$ and $f(\x) = 1$. Note that $T_{x_i = 1}$ corresponds also to the true set of $f_{x_i = 1}$.
\end{definition}

\begin{remark}
Since the reductions of $f$ are identified by its behavior on part of the domain, positivity and smallness are preserved under reduction.
\end{remark}

The following properties classify the behavior of reductions.
\begin{proposition}\label{red.dualpair.prop}
Each $i$-reduction pair are a dual pair, that is, 
\[
	f_{x_i = 1} = f_{x_i = 0}^d.
\]
\begin{proof}
Fix a negation pair $\x, \nx \in \{0, 1\}^n$. Note that since $f$ is self-dual, every negation pair on $\{0, 1\}^{n+1}$ must be given opposite values by $f$.

There are, without loss of generality, three cases for the values of $f_{x_i = 0}$ on $\x, \nx$. If $f_{x_i = 0}(\x) = 0 = f_{x_i = 0}(\nx)$, then $f_{x_i = 1}(\x) = 1$ because it is the negation of $f_{x_i = 0}(\nx)$ when viewed as an element of $f$. Similarly, $f_{x_i = 1}(\nx) = 1$. Thus this negation pair behaves consistently with $f_{x_i =  0}$ and $f_{x_i = 1}$ being dual. The other two cases are identical. Thus $f_{x_i = 1} = f_{x_i = 0}^d$.
\end{proof}
\end{proposition}

\begin{proposition}\label{chow.red.prop}
The Chow parameters for $f_{x_i = 1}$ are given by \[(a_i, a_0, \ldots, a_{i-1}, a_{i+1}, \ldots, a_n).\]
\begin{proof}
Let the Chow parameters for $f_{x_i = 1}$ be denoted by $(b_i, b_0, \ldots, b_{i-1}, b_{i+1}, \ldots, b_n)$.

First we prove that $b_i = a_i$. By definition we have:
\[
	2a_i = |T_{x_i = 1}| + |F_{x_i = 0}|.
\]
But since the assignments are dual, the true set of $f_{x_i = 1}$ has the same cardinality as the false set of $f_{x_i = 0}$, and we have
\[
	2a_i = 2|T_{x_i = 1}| = 2b_i.
\]
Thus the true set of $f_{x_i = 1}$ has cardinality $a_i$. 

Now we prove that $b_j = a_j$ for all $j \neq i$. By definition and a suitable partitioning of the domain, we have:
\begin{align*}
	2a_j &= |T_{x_j = 1}|+|F_{x_j = 0}| \\
    	&= \left( |T_{x_i = 1, x_j = 1}| + |T_{x_i = 0, x_j = 1}|\right) + \left( |F_{x_i = 0, x_j = 0}| + |F_{x_i = 1, x_j = 0}|\right) \\
        &= \left( |T_{x_i = 1, x_j = 1}| + |F_{x_i = 1, x_j = 0}|\right) + \left( |T_{x_i = 0, x_j = 1}| + |F_{x_i = 0, x_j = 0}| \right).
\end{align*}
The left term is $b_j$ by definition, and the right term is the corresponding term in the dual, which is also $b_j$ by Proposition \ref{chow.dual.prop}. Thus we have $a_j = b_j$ for all $j \neq i$.
\end{proof}
\end{proposition}

We now have the tools to prove our critical lemma.
\begin{lemma}\cite{Winder61}\label{canonical.generate.lem}
Given a self-dual representative $f :\{0, 1\}^{n+1} \to \{0, 1\}$ for an SD class, every element of that SD class which is a function on $n$ variables can be reached by the following process:
\begin{enumerate}[label=(\roman*)]
\item setting a single argument of $f$ to either 1 or 0,
\item permuting some remaining arguments of $f$,
\item $\vu$-complementing some remaining arguments of $f$, 
\end{enumerate}
and every LTF so obtained is in the SD class of $f$.
\begin{proof}
Recall that Chow parameters are sufficient invariants to distinguish linear threshold functions (Corollary \ref{chowunique}), and that every element of the SD class is reachable by some sequence of self-dualizations, anti-self-dualizations, permutations, negations, and duals. Consider any such sequence which begins with a self-dual function $f$ on $n+1$ variables and ends with a function $h$ on $n$ variables. Let $f$ have Chow parameters $(a_{-1} = 2^n, a_0, \ldots, a_n)$ and $h$ have Chow parameters $(b_0, \ldots, b_n)$. By the properties of Chow parameters under these operations (Propositions \ref{chow.dual.prop},  \ref{chow.ucomp.prop}, \ref{chow.perm.prop}, \ref{chow.sdual.prop}, \& \ref{chow.red.prop}), every resulting $b_i$ will satisfy either $b_i = \frac{a_j}{2}$ or $b_i = 2^n - \frac{a_j}{2}$ for some $a_j$. Any such sequence of Chow parameters is (1) in the same SD class as $f$ and (2) obtainable by (i)-(iii). Thus an LTF is in the SD class of $f$ if and only if it is obtainable by (i)-(iii). 
\end{proof}
\end{lemma}

\subsection{Proof of validity of Algorithm \ref{algorithm1}}

We are now prepared to prove the validity of Algorithm \ref{algorithm1}. Throughout this section, let $f$ be a canonical generator of its SD class in the sense of Winder. That is, $f$ is a positive, self-dual linear threshold function on $n+1$ variables such that the Chow parameters of $f$, denoted $(2^n, 2^{n+1} - 2a_0, \ldots, 2a_n)$, are in monotonic order. We observe the following corollary of Lemma \ref{canonical.generate.lem}.

\begin{lemma}\label{sd.decomp.lem}
Let $O_{S_n}(h)$ denote the orbit of $h$ under all permutations of arguments, the ``P-type class'' of $h$. Then the set of Goldilocks functions in the SD class of $f$ is given by
$$ \{h \text{ Goldilocks} \; | \; h \in \SD(f)\} = \bigcup_{\fxs \text{ small and ample}} O_{S_n}(\fxs)$$
where the union is taken over all $(i, s) \in \{0,\ldots, n\} \times \{0, 1\}$ such that the corresponding reduction $f_{x_i = s}$ is ample and small. 

Similarly, the set of Semi-Goldilocks functions in the SD class of $f$ is given by
$$ \{h \text{ Semi-Goldilocks} \; | \; h \in \SD(f)\} = \bigcup_{\fxs \text{small}} O_{S_n}(\fxs)$$
where the union is over all $(i, s) \in \{0, \ldots, n\} \times \{0, 1\}$ such that the corresponding reduction is small.
\begin{proof}
For one direction, let $h$ be Goldilocks and in the SD class of $f$. By Lemma \ref{canonical.generate.lem}, we know that $h$ is the $\vu$-complementation of some element of $O_{S_n}(\fxs)$ for some reduction $\fxs$. However, all such elements are positive, since permutation and reduction preserve positivity. Since $h$ is positive, the $\vu$-complementation must be trivial. Thus $h \in O_{S_n}(\fxs)$ for some $\fxs$. Since $h$ is ample and small, it must have been the permutation of some ample and small $\fxs$. For the opposite direction, if $h \in \bigcup O_{S_n}(\fxs)$ where the union is over all ample and small reductions, then $h$ must be the permutation of some positive, ample, and small reduction $\fxs$ and is itself Goldilocks and in the SD class of $f$.

The Semi-Goldilocks proof is identical: since every element in $\bigcup O_{S_n} (f_{x_i = s})$ is a permutation of a Semi-Goldilocks reduction, it must itself be Semi-Goldilocks. For the other direction, by Lemma \ref{canonical.generate.lem} any Semi-Goldilocks $h$ is generated by a reduction, a $\vu$-complementation which must be trivial, and a permutation. Thus $h$ is the permutation of some small reduction $f_{x_i = s}$, proving the opposite inclusion.
\end{proof}
\end{lemma}

This decomposition can be easily written as a union of disjoint permutation orbits with the help of a simple proposition.
\begin{proposition}
If $a_i = a_j$, then $\fxs = f_{x_j = s}$ and the corresponding permutation orbits are identical. Otherwise, the orbits are disjoint. 
\begin{proof}
In the case that $a_i = a_j$, we note that Proposition \ref{chow.red.prop} implies that $\fxs$ and $f_{x_j = s}$ must have identical Chow parameters, and thus must be identical functions. Clearly, this implies that they have the same permutation orbit.

If $a_i \neq a_j$, then (also by Proposition \ref{chow.red.prop}) $\fxs$ and $f_{x_j = s}$ will have different sets of Chow parameters $\{a_k \; | \; k \neq i \} \neq \{a_k \; | \; k \neq j\}$. Thus they cannot be permutations of one another, and must have disjoint orbits.
\end{proof}
\end{proposition}

\begin{corollary}\label{count.gold.sd.cor}
The number of Semi-Goldilocks functions in the SD class of $f$ is given by 
$$\SGold^{\SD}(n) = \sum |O_{S_n}(f_{x_i = s})|$$
where the summation is over all distinct small reductions $f_{x_i = s}$.

Similarly, the number of Goldilocks functions in the SD class of $f$ is given by 
$$\Gold^{\SD}(f) = \sum |O_{S_n}(f_{x_i = s})|$$
where the summation is taken over all distinct ample, small reductions.
\end{corollary}

We can now prove the validity of the main subroutines.
\begin{lemma}\label{alg.semivalid.lem}
If $f$ has Chow parameters $(2^n, 2\vec{a})$, then \textproc{SGoldSD}$(f, \vec{a})$ returns the number of Semi-Goldilocks functions in the SD class of $f$.
\begin{proof}
Lines \ref{alg.sg.ired}-\ref{alg.sg.endloop} iterate through all of the $i$-reduction pairs of $f$. Since the Chow parameters are in monotonic order, the condition on \ref{alg.sg.dist} ensures that only new reductions are considered. Thus, the core of the loop iterates once for each distinct reduction pair.

If the first reduction $f_{x_i = 1}$ is small, then all of its permutations are counted (\ref{alg.sg.1small}, \ref{alg.sg.1add}). The second reduction $f_{x_i = 0}$ is tested identically (\ref{alg.sg.0small}), but only if the pair are not self-dual (to prevent double counting). Thus, the loop adds to $s$ the number of semi-Goldilocks LTFs which can be reached from $f_{x_i = 0}$ or $f_{x_i = 1}$ by permutations. Thus, after line \ref{alg.sg.endloop}, $s$ holds the sum of the size of permutation orbits of all small reductions. Thus by \ref{count.gold.sd.cor}, \textproc{SGoldSD}$(f, \vec{a})$ returns the number of Semi-Goldilocks functions in the SD class of $f$.
\end{proof}
\end{lemma}

\begin{lemma}\label{alg.valid.lem}
If $f$ has Chow parameters $(2^n, 2\vec{a})$, then \GoldSDa returns the number of Goldilocks functions in the SD class of $f$.
\begin{proof}
Lines (3)-(13) iterate over all $i$-reduction pairs. Since the Chow parameters $\vec{a}$ are in monotonic order, the condition on (4) returns true if and only if this reduction is distinct from those seen before. Thus lines (5)-(11) execute exactly once for each distinct reduction.

Since the $i$-reductions are dual pairs (Proposition \ref{red.dualpair.prop}), only one can be ample (unless self-dual, in which case we can consider either). By Propositoin \ref{chow.dual.prop}, the condition on line (6) holds if and only if $f_{x_i = 0}$ is not ample, in which case $f_{x_i = 1}$ must be. Thus, after (8), $\val$ holds the correct value such that $f_{x_i = \val}$ is ample. Since \textproc{IsSmall} correctly tests for smallness and \textproc{PermCount} returns the size of the permutation orbit 
, line (10) adds the size of the permutation orbit if and only if $f_{x_i = \val}$ is small and ample.

Since the size of the orbit of each distinct ample, small reduction is added exactly once to $s$, Corollary \ref{count.gold.sd.cor} implies that after line (14) executes, $s$ holds exactly the number of Goldilocks functions in the SD class of $f$.
\end{proof}
\end{lemma}

\begin{theorem}\label{alg.semivalid.thm}
The sum of \textproc{SGoldSD}$(f, \vec{a})$ across all canonical generators $f$ output by Winder's algorithm is the total number of Semi-Goldilocks functions on $n$ variables.
\begin{proof}
By definition, the SD classes partition the set of LTFs, so the Goldilocks functions can be split across the SD classes. Winder's algorithm \cite{Winder61} generates a single representative of each class which contains any LTFs on $n$ variables, in the form of a self-dual, positive LTF with Chow parameters in monotonic order. 

Lemma \ref{alg.semivalid.lem} implies that \textproc{SGoldSD}$(f, \vec{a})$ on such arguments returns the number of Goldilocks functions in the SD class of $f$, so taking the total of \textproc{SGoldSD} across all generators gives exactly $\SGold(n)$.

\end{proof}
\end{theorem}

\begin{theorem}\label{alg.valid.thm}
The sum of \GoldSD{f, \vec{a}} across all canonical generators $f$ output by Winder's algorithm is the total number of Goldilocks functions on $n$ variables.
\begin{proof}
The argument is identical to the Semi-Goldilocks case but applies the validity of \GoldSDa (\ref{alg.valid.lem}).
\end{proof}
\end{theorem}

\bibliographystyle{alpha}

\bibliography{bib}

\end{document}